\documentclass[11pt]{amsart}
\usepackage{amsmath}
\usepackage{amssymb}
\usepackage{tabularx}
\usepackage{enumerate}
\usepackage{graphicx}
\usepackage{texdraw}

\usepackage{mathrsfs}

\usepackage{amsfonts,amssymb,amsmath}
\usepackage{epsfig}

\makeatletter
\@addtoreset{equation}{section}

\makeatother
\newtheorem{thm}{Theorem}[section]
\newtheorem{lem}[thm]{Lemma}
\newtheorem{cor}[thm]{Corollary}
\newtheorem{prop}[thm]{Proposition}
\newtheorem{remark}[thm]{Remark}

\newcommand{\R}{\mathbb{R}}

\begin{document}
\title[Spreading and Vanishing in Nonlinear Diffusion Problems]{Spreading
and Vanishing in Nonlinear Diffusion Problems with
Free Boundaries$^\S$}
 \thanks{$\S$ This research was partly supported by the Australian Research
Council, by the NSFC (10971155) and by Innovation Program of
Shanghai Municipal Education Commission (09ZZ33). }
\author[Y. Du and B. Lou]{Yihong Du$^\dag$ and Bendong Lou$^\ddag$}
\thanks{$\dag$ School of Science and
Technology, University of New England, Armidale, NSW 2351,
Australia.}
\thanks{$\ddag$ Department of Mathematics, Tongji
University, Shanghai 200092, China.}
\thanks{{\bf Emails:} {\sf ydu@turing.une.edu.au} (Y. Du), {\sf
blou@tongji.edu.cn} (B. Lou) }
\date{July 16, 2012}

\begin{abstract} We study nonlinear diffusion problems of the
form $u_t=u_{xx}+f(u)$ with free boundaries.
 Such problems may be used to describe the spreading of a
biological or chemical species, with the free boundary representing
the expanding front. For special  $f(u)$ of the Fisher-KPP type, the
problem was investigated by Du and Lin \cite{DuLin}. Here we
consider much more general nonlinear terms. For any $f(u)$ which is
$C^1$ and satisfies $f(0)=0$,  we show that the omega limit set
$\omega(u)$ of every bounded positive solution is determined by   a
stationary solution. For monostable, bistable and combustion types
of nonlinearities, we obtain a rather complete description of the
long-time dynamical behavior of the problem; moreover, by
introducing a parameter $\sigma$ in the initial data, we reveal a
threshold value $\sigma^*$ such that spreading ($\lim_{t\to\infty}u=
1$) happens when $\sigma>\sigma^*$,
 vanishing ($\lim_{t\to\infty}u=0$) happens when
$\sigma<\sigma^*$,  and at the threshold value $\sigma^*$,
$\omega(u)$ is different for the three different types of
nonlinearities. When spreading happens, we make use of
``semi-waves'' to determine the asymptotic spreading speed of the
front.
\end{abstract}

 \maketitle

\section{Introduction}

 We consider the following problem
\begin{equation}\label{p}
\left\{
\begin{array}{ll}
 u_t = u_{xx} + f(u), &  g(t)< x<h(t),\ t>0,\\
 u(t,g(t))= u(t,h(t))=0 , &  t>0,\\
g'(t)=-\mu\, u_x(t, g(t)), & t>0,\\
 h'(t) = -\mu \, u_x (t, h(t)) , & t>0,\\
-g(0)=h(0)= h_0,\ \ u(0,x) =u_0 (x),& -h_0\leq x \leq h_0,
\end{array}
\right.
\end{equation}
where $x=g(t)$ and $x=h(t)$ are the moving boundaries to be
determined together with $u(t,x)$, $\mu$ is a given positive
constant, $f:[0,\infty)\rightarrow \R$ is a $C^1$ function
satisfying
\begin{equation}\label{f(0)}
f(0)=0.
\end{equation}
The initial function $u_0$ belongs to  $ \mathscr {X}(h_0)$ for some
$h_0>0$, where
\begin{equation}\label{def:X}
\begin{array}{ll}
\mathscr {X}(h_0):= \Big\{ \phi \in C^2 ([-h_0,h_0]): & \phi(-h_0)=
\phi (h_0)=0,\; \phi'(-h_0)>0,\\ & \phi'(h_0)<0,\;
 \phi(x) >0 \ \mbox{in } (-h_0,h_0)\;\Big\}.
\end{array}
\end{equation}
For any given $h_0>0$ and $u_0 \in \mathscr {X}(h_0)$, by a
(classical) solution of \eqref{p} on the time-interval $[0,T]$ we
mean a triple $(u(t,x), g(t), h(t))$ belonging to  $ C^{1,2}(G_T)
\times C^1 ([0,T])\times C^1 ([0,T])$, such that all the identities
in \eqref{p} are satisfied pointwisely, where
\[
G_T:=\big\{(t,x): t\in (0,T],\; x\in [g(t), h(t)]\big\}.
\]
In the rest of the paper,  the solution may also be denoted by
$(u(t,x; u_0), g(t; u_0), h(t; u_0))$,  or simply $(u,g,h)$,
depending on the context.

Problem \eqref{p} with $f(u)$ taking the particular form $f(u)=au-b
u^2$ was studied recently in \cite{DuLin}. Such a situation arises
as a population model describing the spreading of a new or invasive
species, whose growth is governed by the logistic law. The free
boundaries $x=g(t)$ and $x=h(t)$ represent the spreading fronts of
the population whose density is represented by $u(t,x)$.  The focus
of \cite{DuLin} is on the particular logistic nonlinearity
$f(u)=au-b u^2$, and many of the arguments there rely on this choice
of $f$.

The logistic $f(u)$ mentioned above belongs to the class of
``monostable'' nonlinearities, and  due to the pioneering works of
Fisher \cite{F} and Kolmogorov-Petrovski-Piskunov \cite{KPP}, it is
also known as the Fisher, or KPP, or Fisher-KPP type nonlinearity.
As is well-known, in population models one often needs to consider a
general monostable nonlinear term (\cite{AW1, AW2}). Moreover, to
include Allee effects, ``bistable'' nonlinear terms are used in many
population models (\cite{HR, LK}). Bistable nonlinearity also
appears in other applications including signal propagation and
material science (\cite{NAY, AC, FM}). Furthermore, in the study of
combustion problems the typical $f(u)$ is of ``combustion'' type
(\cite{ZFK, K, Zla}). A precise description of these different types
of nonlinearities will be given shortly below.

The main purpose of this paper is to classify the behavior of
\eqref{p} for all the types of nonlinearities mentioned in the last
paragraph. Even restricted to the monostable type, this is an
extension of \cite{DuLin} since we do not require the special form
$f(u)=au-b u^2$, which implies that different methods have to be
used.

The corresponding Cauchy problem
\begin{equation}
\label{Cauchy} u_t=u_{xx}+f(u) \; (x\in\R^1,\; t>0),\; u(0,x)=u_0(x)
\; (x\in \R^1)
\end{equation}
has been extensively studied. For example, the classical paper
\cite{AW1} contains a systematic investigation of this problem.
Various sufficient conditions for $\lim_{t\to \infty} u(t,x)=1$ and
for $\lim_{t\to \infty} u(t,x)=0$ are known, and when $u_0(x)$ is
nonnegative and has compact support, the way $u(t,x)$ approaches 1
as $t\to\infty$ was used to describe the spreading of a (biological
or chemical) species, which is characterized by certain traveling
waves, and the speed of these traveling waves determines the
asymptotic spreading speed of the species; see for example, \cite{K,
FM, AW1, AW2}. The transition between spreading ($u\to 1$) and
vanishing ($u\to 0$) has not been well understood until recently. In
\cite{DM}, motivated by break-through results obtained in
\cite{Zla}, a rather complete description of the sharp transition
behavior was given. As we will see below, these sharp transition
results of \cite{DM} for \eqref{Cauchy} also hold for \eqref{p}.
(Spreading and vanishing are sometimes called propagation and
extinction, as in \cite{DM}.) We will make use of a number of the
ideas from \cite{DM}, and this paper may be regarded as an extension
of \cite{DM}.

In most spreading processes in the natural world, a spreading front
can be observed. In the one space dimension case, if the species
initially occupies an interval $(-h_0,h_0)$ with density $u_0(x)$,
as time $t$ increases from 0, it is natural to expect the two end
points of $(-h_0,h_0)$ to evolve into two spreading fronts, $x=g(t)$
on the left and $x=h(t)$ on the right, and the initial function
$u_0(x)$ to evolve into a positive function $u$ inside the interval
$(g(t), h(t))$ governed by the equation $u_t=u_{xx}+f(u)$, with $u$
vanishing at $x=g(t)$ and $x=h(t)$. To determine how the fronts
$x=g(t)$ and $x=h(t)$ evolve with time, we assume that the fronts
invade at a speed that is proportional to the spatial gradient of
the density function $u$ there, which gives rise to the free
boundary conditions in \eqref{p}. A deduction of this free boundary condition based on 
ecological assumptions  can be found in \cite{BDK}.

We notice that the free boundary conditions in \eqref{p} coincide with the one-phase Stefan condition
arising from the investigation of the melting of ice in contact with
water (\cite{Ru}). Such conditions also arise in the modeling of
wound healing (\cite{CF}). For population models, \cite{Lin} used
such a condition for a predator-prey system over a bounded interval,
showing the free boundary reaches the fixed boundary in finite time,
and hence the long-time dynamical behavior of the system is the same
as the well-studied fixed boundary problem; and in \cite{MYY1}, a
two phase Stefan condition was used for a competition system over a
bounded interval, where the free boundary separates the two
competitors from each other in the interval. A similar problem to
\eqref{p} but with $f(u)=u^p \; (p>1)$ was studied in \cite{FS,
GST}. Since this is a superlinear problem, its behavior is very
different from \eqref{p} considered here as our focus is on the
sublinear cases (except Theorem \ref{thm:convergence} and section
2). Indeed, our interests here are very different from all the
previous research mentioned in this paragraph.

We now describe the main results of this paper. Firstly we assume
that
\begin{equation}
\label{cond1} \mbox{ $f(u)$ is $C^1$ and  $f(0)=0$.}
\end{equation}
 Then a simple variation of the arguments in \cite{DuLin} shows
that, for any $h_0>0$ and $u_0\in  \mathscr {X}(h_0)$, \eqref{p} has
a unique solution defined on some maximal time interval $(0, T^*)$,
$T^*\in (0, \infty]$. Moreover, $g'(t)<0,\; h'(t)>0$ and $u(t,x)>0$
for $t\in (0, T^*),\; x\in (g(t),h(t))$, and if $T^*<\infty$ then
$\max_{x\in [g(t), h(t)]} u(t,x)\to \infty$ as $t\to T^*$. Thus
$\lim_{t\to\infty}g(t)$ and $\lim_{t\to\infty}h(t)$ always exist if
$T^*=\infty$. Throughout this paper, we will use the notations
\[
g_\infty:=\lim_{t\to\infty}g(t),\;\;
h_\infty:=\lim_{t\to\infty}h(t).
\]
$T^*=\infty$ is guaranteed if we assume further that
\begin{equation}
\label{cond2} f(u)\leq Ku \mbox{ for all $u\geq 0$ and some $K>0$}.
\end{equation}
A more detailed description of these statements can be found in
section 2 below.

\bigskip

Our first main result is a general convergence theorem, which is an
analogue of Theorem 1.1 in \cite{DM}.

\begin{thm}
\label{thm:convergence} Suppose that \eqref{cond1} holds and $(u,
g,h)$ is a solution of \eqref{p} that is defined for all $t>0$, and
$u(t,x)$ is bounded, namely
\[
u(t,x)\leq C \mbox{ for all $t>0$, \ $x\in [g(t), h(t)]$ and some
$C>0$.}
\]
Then $(g_\infty, h_\infty)$ is either a finite interval or
$(g_\infty, h_\infty)=\R^1$. Moreover, if $(g_\infty, h_\infty)$ is
 a finite interval, then
$\lim_{t\to\infty} u(t,x)=0$, and if $(g_\infty, h_\infty)=\R^1$
then either $\lim_{t\to\infty} u(t,x)$ is a nonnegative constant
solution of
\begin{equation}
\label{ellip}
 v_{xx}+f(v)=0,\; x\in \R^1,
 \end{equation}
 or
 \[
 u(t,x)-v(x+\gamma(t))\to 0 \mbox{ as } t\to\infty,
 \]
 where $v$ is an evenly decreasing positive solution
 of \eqref{ellip}, and $\gamma: [0,\infty)\to [-h_0,h_0]$ is a
 continuous function.
\end{thm}

By an evenly decreasing function we mean a function $v(x)$
satisfying $v(-x)=v(x)$ which is strictly decreasing in
$[0,\infty)$. Let us note that $(g_\infty, h_\infty)$ can never be a
half-infinite interval. In fact, we will prove in Lemma
\ref{lem:center} that
\[
-2h_0<g(t)+h(t)<2h_0 \mbox{ for all } t>0.
\]
We conjecture that $\lim_{t\to\infty} \gamma(t) $ exists but were
unable to prove it.

\bigskip

Next we focus on three types of nonlinearities: \vskip 6pt
\begin{center}
(f$_M$)\ \ monostable case, \ \ \ (f$_B$)\ \ bistable case,\ \ \
(f$_C$)\ \ combustion case.
\end{center}

In the monostable case (f$_M$), we assume that $f$ is $C^1$ and it
satisfies
\begin{equation}\label{mono}
f(0)=f(1)=0, \quad f'(0)>0, \quad  f'(1)<0,\quad (1-u)f(u) >0 \ \mbox{for } u>0, u\not= 1.
\end{equation}
 Clearly $f(u) =u(1-u)$ belongs to (f$_M$).

In the bistable case (f$_B$), we assume that $f$ is $C^1$ and it
satisfies
\begin{equation}\label{bi}
f(0)=f(\theta)= f(1)=0, \quad f(u) \left\{
\begin{array}{l}
<0 \ \ \mbox{in } (0,\theta),\\
>0\ \  \mbox{in } (\theta, 1),\\
< 0\ \ \mbox{in } (1,\infty)
\end{array} \right.
\end{equation}
for some $\theta\in (0,1)$,  $f'(0)<0$, $f'(1)<0$ and
\begin{equation}\label{unbalance}
\int_0^1 f(s) ds >0.
\end{equation}
A typical bistable $f(u)$ is $u(u-\theta)(1-u)$ with $\theta \in (0,
\frac{1}{2})$.

In the combustion case (f$_C$), we assume that $f$ is $C^1$ and it
satisfies
\begin{equation}\label{combus}
f(u)=0 \ \ \mbox{in } [0,\theta], \quad f(u) >0 \ \mbox{in }
(\theta,1), \quad f'(1)<0,\quad f(u) < 0 \ \mbox{in } [1, \infty)
\end{equation}
for some $\theta \in (0,1)$, and there exists a small $\delta_0  >0$
such that
\begin{equation}\label{combus1}
f(u)\ \mbox{is nondecreasing in } (\theta, \theta+\delta_0).
\end{equation}

Clearly  \eqref{cond1} and \eqref{cond2} are satisfied if $f$ is of
{\rm (f$_M$)}, or {\rm (f$_B$)}, or {\rm (f$_C$)} type. Thus in
these cases \eqref{p} always has a unique solution defined for all
$t>0$.

\vskip 6pt

The next three theorems give a rather complete description of the
long-time behavior of the solution, and they also reveal the related
but different sharp transition natures between vanishing and
spreading for these three types of nonlinearities.

\begin{thm}\label{thm:mono}
{\rm (The monostable case).}  Assume that $f$ is of {\rm (f$_M$)}
type, and $h_0>0$, $u_0 \in \mathscr {X}(h_0)$. Then either

{\rm (i) Spreading:} $(g_\infty, h_\infty)=\R^1$ and
\[
\lim_{t\to\infty}u(t,x)=1 \mbox{ locally uniformly in $\R^1$},
\]
or

{\rm (ii) Vanishing:} $(g_\infty, h_\infty)$ is a finite interval
with length no bigger than $\pi/\sqrt{f'(0)}$ and
\[
\lim_{t\to\infty}\max_{g(t)\leq x\leq h(t)} u(t,x)=0.
\]
Moreover, if $u_0=\sigma \phi$ with $\phi\in \mathscr {X}(h_0)$,
then there exists $\sigma^* = \sigma^* (h_0, \phi) \in [0,\infty]$
such that vanishing happens when $ 0< \sigma \leq \sigma^*$, and
spreading happens when $ \sigma > \sigma^*$. In addition,
\[
\sigma^* \left\{\begin{array}{ll} =0 &\mbox{ if $h_0 \geq \pi/(
2\sqrt{f'(0)} )$,}\\
  \in (0, \infty] & \mbox{ if $h_0 < \pi/(
2\sqrt{f'(0)} )$,}\\
  \in (0, \infty) &\mbox{ if $h_0 < \pi/(
2\sqrt{f'(0)} )$ and if $f$ is globally Lipschitz.}
\end{array}
\right.
\]
\end{thm}

\begin{thm}\label{thm:bi}
{\rm (The bistable case).} Assume that $f$ is of {\rm (f$_B$)} type,
and  $h_0>0$, $u_0 \in \mathscr {X}(h_0)$. Then either

{\rm (i) Spreading:} $(g_\infty, h_\infty)=\R^1$ and
\[
\lim_{t\to\infty}u(t,x)=1 \mbox{ locally uniformly in $\R^1$},
\]
or

{\rm (ii) Vanishing:} $(g_\infty, h_\infty)$ is a finite interval
and
\[
\lim_{t\to\infty}\max_{g(t)\leq x\leq h(t)} u(t,x)=0,
\]
or

{\rm (iii) Transition:} $(g_\infty, h_\infty)=\R^1$ and there exists
a continuous function $\gamma: [0,\infty)\to [-h_0,h_0]$ such that
\[
\lim_{t\to\infty}|u(t,x)-v_\infty(x+\gamma(t))|=0 \mbox{ locally
uniformly in $\R^1$},
\]
where $v_\infty$ is the unique positive solution to
\[v''+f(v)=0 \; (x\in\R^1),\; v'(0)=0,\; v(-\infty)=v(+\infty)=0.\]
 Moreover, if $u_0=\sigma \phi$ for some $\phi\in \mathscr {X}(h_0)$,
then there exists $\sigma^* = \sigma^* (h_0, \phi)\in (0,\infty]$
such that vanishing happens when $ 0<\sigma < \sigma^*$, spreading
happens when $\sigma>\sigma^*$, and transition happens when
$\sigma=\sigma^*$. In addition, there exists $Z_B>0$ such that
$\sigma^* <\infty$ if $h_0 \geq Z_B$, or if $h_0 < Z_B$ and $f$ is
globally Lipschitz.
\end{thm}

\begin{thm}\label{thm:combus}
{\rm (The combustion case).} Assume that $f$ is of {\rm (f$_C$)}
type, and $h_0>0$, $u_0 \in \mathscr {X}(h_0)$. Then either

{\rm (i) Spreading:} $(g_\infty, h_\infty)=\R^1$ and
\[
\lim_{t\to\infty}u(t,x)=1 \mbox{ locally uniformly in $\R^1$},
\]
or

{\rm (ii) Vanishing:} $(g_\infty, h_\infty)$ is a finite interval
and
\[
\lim_{t\to\infty}\max_{g(t)\leq x\leq h(t)} u(t,x)=0,
\]
or

{\rm (iii) Transition:} $(g_\infty, h_\infty)=\R^1$ and
\[
\lim_{t\to\infty}u(t,x)=\theta \mbox{ locally uniformly in $\R^1$}.
\]
Moreover, if $u_0=\sigma \phi$ for some $\phi\in \mathscr {X}(h_0)$,
then there exists $\sigma^* = \sigma^* (h_0, \phi)\in (0,\infty]$
such that vanishing happens when $ 0< \sigma < \sigma^*$, spreading
happens when $\sigma>\sigma^*$, and transition happens when
$\sigma=\sigma^*$. In addition, there exists $Z_C>0$ such that
$\sigma^* <\infty$ if $h_0 \geq Z_C$, or if $h_0 < Z_C$ and $f$ is
globally Lipschitz.
\end{thm}

\begin{remark}\rm
The value of $\sigma^*$ in the above theorems can be $+\infty$ if we
drop the assumption that $f$ is globally Lipschitz when $h_0$ is
small. Indeed, this is the case if $f(u)$ goes to $-\infty$ fast
enough as $u\to+\infty$; see Propositions
\ref{prop:mono-sigma-infty}, \ref{prop:bi-sigma-infty} and
\ref{prop:combus-sigma-infty}  for details. The values of $Z_B$ and
$Z_C$ are determined by \eqref{def:RB1} and \eqref{def:RC1},
respectively.
\end{remark}
\begin{remark}\rm
In \cite{DuLin}, to determine whether spreading or vanishing happens
for the special monostable nonlinearity, a threshold value of $\mu$
was established, which was shown in \cite{DuLin} to be always
finite. Here we use $\sigma$ in $u_0=\sigma\phi$ as a varying
parameter, which appears more natural especially for the bistable
and combustion cases, since in these cases the dynamical behavior of
\eqref{p} is more responsive to the change of the initial function
than to the change of $\mu$; for example, when $\|u_0\|_\infty\leq
\theta$, then vanishing always happens regardless of the value of
$\mu$.
 \end{remark}
 \begin{remark} \rm Theorems \ref{thm:bi} and \ref{thm:combus} above are
 parallel to Theorems 1.3 and 1.4 in \cite{DM}, where the Cauchy problem
 was considered. In contrast, Theorem \ref{thm:mono} is very
 different from the Cauchy problem version, where a ``hair-trigger''
 phenomenon appears, namely, when $f$ is of (f${_M}$) type,  any nonnegative
 solution of \eqref{Cauchy} is either identically 0,  or it converges to 1
 as $t\to\infty$ (see \cite{AW1, AW2}).
\end{remark}

\vskip 6pt

When spreading happens, the asymptotic spreading speed is determined
by the following problem
\begin{equation}\label{prop-profile}
\left\{
  \begin{array}{l}
  q_{zz} - c q_z + f(q) =0\ \ \mbox{ for }  z\in (0,\infty),\\
  q(0)=0, \; \mu q_z(0) = c,\; q(\infty)=1,\; q(z)>0 \mbox{ for } z>0.
  \end{array}
  \right.
\end{equation}

\begin{prop}
\label{prop:semi-wave} Assume that $f$ is of {\rm (f$_M$)}, or {\rm
(f$_B$)}, or {\rm (f$_C$)} type. Then for each $\mu>0$,
\eqref{prop-profile}
 has a unique solution $(c,q)=(c^*, q^*)$.
 \end{prop}

 We call $q^*$ a ``semi-wave'' with speed $c^*$, since
 the function $v(t,x)=q^*(c^*t-x)$ satisfies
 \[
 v_t=v_{xx}+f(v) \; (t\in\R^1,\; x<c^*t),\; v(t, c^* t)=0,\;
 v(t,-\infty)=1,
 \]
 and it resembles a wave moving to the right at constant speed
 $c^*$,  with front at $x=c^*t$. In comparison with the normal
 traveling wave generated by the solution of
\begin{equation}
\label{tw} q_{zz}-cq_z+f(q)=0 \mbox{ for } z\in\R^1,\;
q(-\infty)=0,\; q(+\infty)=1,
\end{equation}
 the generator $q^*(z)$ of $v(t,x)$ here is only defined on the half line $\{z\geq 0\}$.
 Hence we call it a semi-wave.
  We notice that
 at the front $x=c^*t$, we have $c^*= -\mu v_x(t, x)$, namely the
 Stefan condition in \eqref{p} is satisfied by $v(t,x)$ at $x=c^*t$.

Making use of the above semi-wave, we can prove the following
result.

\begin{thm}\label{thm:spreading speed}
Assume that $f$ is of {\rm (f$_M$)}, or {\rm (f$_B$)}, or {\rm
(f$_C$)} type, and spreading happens. Let $c^*$ be given by
Proposition \ref{prop:semi-wave}. Then
\[
\lim_{t\to\infty}\frac{h(t)}{t}=\lim_{t\to\infty}
\frac{-g(t)}{t}=c^*,
\]
and for any small $\varepsilon
>0$, there exist positive constants $\delta$, $M$ and $T_0$ such that
\begin{equation}\label{est-u}
\max\limits_{|x| \leq (c^* -\varepsilon) t} |u(t,x) -1|\leq
Me^{-\delta t} \mbox{ for all } t\geq T_0.
\end{equation}
\end{thm}

\begin{remark}\rm
The asymptotic spreading speed $c^*$ depends on the parameter $\mu$
appearing in the free boundary conditions and in
\eqref{prop-profile}. Therefore we may denote $c^*$ by $c^*_\mu$ to
stress this dependence. It is well-known {\rm (}see, e.g.,
\cite{AW1, AW2}{\rm )} that when $f$ is of {\rm (f$_M$)}, or {\rm
(f$_B$)}, or {\rm (f$_C$)} type, the asymptotic spreading speed
determined by the Cauchy problem \eqref{Cauchy} is given by the
speed of certain traveling wave solutions generated by a solution of
\eqref{tw}. Let us denote this speed by $c_0$. Then we have (see
Theorem \ref{waves}): $c_\mu^*$ is increasing in $\mu$ and
\[
\lim_{\mu\to\infty}c^*_\mu=c_0.
\]
\end{remark}
\begin{remark}\rm
It is possible to show that the Cauchy problem \eqref{Cauchy} is the
limiting problem of \eqref{p} as $\mu\to\infty$. This holds in much
more general situations; see section 5 of \cite{DuGuo2} for the
general higher space dimension case.
\end{remark}

\bigskip

The rest of the paper is organized as follows. In section 2, we
present some basic results which are fundamental for this research,
and may have other applications. Here we only assume that $f$ is
$C^1$ and $f(0)=0$, namely \eqref{cond1} holds. The proofs of some
of these results are modifications of existing ones. Firstly we give
two comparison principles formulated in  forms that are convenient
to use in this paper. Secondly we explain how the arguments in
\cite{DuLin} can be modified to show the uniqueness and existence
result for \eqref{p} under \eqref{cond1}. Thirdly we give the proof
of Theorem \ref{thm:convergence}. This is based on a key fact proved
in Lemma \ref{lem:center}, which  says that the solution is rather
balanced in $x$ as it evolves with time $t$, though it is not
symmetric in $x$ in general.  The rest of the proof
 largely follows the approach in \cite{DM}.

In section 3, for monostable, bistable and combustion
nonlinearities, we give a number of sufficient conditions for
vanishing (see Theorem \ref{thm:vanishing}), through the
construction of suitable upper solutions.

In section 4, we obtain sufficient conditions for spreading for the
three types of nonlinearities. This is achieved by constructing
suitable lower solutions based on a phase plane analysis of the
equation
\[
q''-cq' +f(q)=0
\]
over  a bounded interval $[0,Z]$, together with suitable conditions
at the ends of this interval.

Section 5 is devoted to the proofs of Theorems \ref{thm:mono},
\ref{thm:bi} and \ref{thm:combus}, with the proof of each theorem
constituting a subsection. The arguments here rely heavily on the
results in the previous sections. The proof of the fact mentioned in
Remark 1.5, namely $\sigma^*=+\infty$ when $f(u)$ goes to $-\infty$
fast enough, is rather technical, and is given in subsection 5.4.

Proposition \ref{prop:semi-wave} and Theorem \ref{thm:spreading
speed} are proved in section 6, the last section of the paper. In
subsection 6.1, we prove Proposition \ref{prop:semi-wave} by
revisiting the well-known traveling wave solution with speed $c_0$
(the minimal speed for monostable type nonlinearity, and unique
speed for nonlinearity of bistable or combustion type). Our phase
plane analysis is related to but different from that in \cite{AW1,
AW2}. This alternative method leads to the desired semi-wave
naturally; see Remark 6.3 for further comments. Subsection 6.2 is
devoted to the proof of Theorem \ref{thm:spreading speed}.

\section{Some Basic Results}\label{sec:basic}

In this section we give some basic results which will be frequently
used later in the paper. The results here are for general $f$ which
is $C^1$ and satisfies $f(0)=0$.

\begin{lem}
\label{lem:comp1} Suppose that \eqref{cond1} holds, $T\in
(0,\infty)$, $\overline g, \overline h\in C^1([0,T])$, $\overline
u\in C(\overline D_T)\cap C^{1,2}(D_T)$ with $D_T=\{(t,x)\in\R^2:
0<t\leq T, \overline g(t)<x<\overline h(t)\}$, and
\begin{eqnarray*}
\left\{
\begin{array}{lll}
\overline u_{t} \geq \overline u_{xx} +f(\overline u),\; & 0<t \leq T,\
\overline g(t)<x<\overline h(t), \\
\overline u= 0,\quad \overline g'(t)\leq -\mu \overline u_x,\quad &
0<t \leq T, \ x=\overline g(t),\\
\overline u= 0,\quad \overline h'(t)\geq -\mu \overline u_x,\quad
&0<t \leq T, \ x=\overline h(t).
\end{array} \right.
\end{eqnarray*}
 If
\[
\mbox{$[-h_0, h_0]\subseteq [\overline g(0), \overline h(0)]$ \quad
and \quad $u_0(x)\leq \overline u(0,x)$ in $[-h_0,h_0]$,}
\]
and $(u,g, h)$ is a solution to \eqref{p}, then
\[
\mbox{ $g(t)\geq \overline g(t),\; h(t)\leq\overline h(t)$ in $(0,
T]$,}
\]
\[
\mbox{$u(x,t)\leq \overline u(x,t)$ for $t\in (0, T]$ \quad and
\quad $x\in (g(t), h(t))$.}
\]
\end{lem}

\begin{lem}
\label{lem:comp2} Suppose that  \eqref{cond1} holds, $T\in
(0,\infty)$, $\overline g,\, \overline h\in C^1([0,T])$, $\overline
u\in C(\overline D_T)\cap C^{1,2}(D_T)$ with $D_T=\{(t,x)\in\R^2:
0<t\leq T, \overline g(t)<x<\overline h(t)\}$, and
\begin{eqnarray*}
\left\{
\begin{array}{lll}
\overline u_{t}\geq  \overline u_{xx}+ f(\overline u),\; &0<t \leq T,\ \overline g(t)<x<\overline h(t), \\
\overline u\geq u, &0<t \leq T, \ x= \overline g(t),\\
\overline u= 0,\quad \overline h'(t)\geq -\mu \overline u_x,\quad
&0<t \leq T, \ x=\overline h(t),
\end{array} \right.
\end{eqnarray*}
with
\[
\mbox{$\overline g(t)\geq g(t) $ in $[0,T]$,\quad $h_0\leq \overline
h(0),$\quad  $u_0(x)\leq \overline u(0,x)$ in $[\overline
g(0),h_0]$,}
\]
where  $(u,g, h)$ is a solution to \eqref{p}. Then
\[
\mbox{ $h(t)\leq\overline h(t)$ in $(0, T]$,\quad $u(x,t)\leq
\overline u(x,t)$ for $t\in (0, T]$\ and $ \overline g(t)<x< h(t)$.}
\]
\end{lem}

The proof of Lemma \ref{lem:comp1} is identical to that of Lemma 5.7
in \cite{DuLin}, and a minor modification of this proof yields Lemma
\ref{lem:comp2}.

\begin{remark}
\label{rem5.8}\rm The function $\overline u$, or the triple
$(\overline u,\overline g,\overline h)$, in Lemmas \ref{lem:comp1}
and \ref{lem:comp2} is often called an upper solution to \eqref{p}.
A lower solution can be defined analogously by reversing all the
inequalities. There is a symmetric version of Lemma~\ref{lem:comp2},
where the conditions on the left and right boundaries are
interchanged. We also have corresponding comparison results for
lower solutions in each case.
\end{remark}

The following local existence result can be proved by the same
arguments as in \cite{DuLin} (see Theorem 2.1 and the beginning of
section 5 there).

\begin{thm}
\label{thm:local} Suppose that \eqref{cond1} holds. For any given
$u_0\in \mathscr {X}(h_0)$ and any $\alpha\in (0,1)$, there is a
$T>0$ such that problem \eqref{p} admits a unique solution
$$(u, g, h)\in C^{(1+\alpha)/2,
1+\alpha}(\overline{G}_{T})\times C^{1+\alpha/2}([0,T])\times C^{1+\alpha/2}([0,T]);$$
moreover, 
\begin{eqnarray}
\|u\|_{C^{
(1+\alpha)/2,
1+\alpha}(\overline{G}_{T})}+\|g\|_{C^{1+\alpha/2}([0,T])}+\|h\|_{C^{1+\alpha/2}([0,T])}\leq C, 
\label{b12}
\end{eqnarray}
where $G_{T}=\{(t,x)\in \R^2: x\in [g(t), h(t)], t\in (0,T]\}$, $C$
and $T$ only depend on $h_0$, $\alpha$ and $\|u_0\|_{C^{2}([-h_0,
h_0])}$.
\end{thm}

\begin{remark}\rm
As in \cite{DuLin}, by the Schauder estimates applied to the
equivalent fixed boundary problem used in the proof, we have
additional regularity for $u$, namely,  $u\in C^{1+\alpha/2,
2+\alpha}(G_T)$.
\end{remark}

\begin{lem}
\label{lem:bound-general} Suppose that \eqref{cond1} holds,
$(u,g,h)$ is a solution to \eqref{p} defined for $t\in [0,T_0)$ for
some $T_0\in (0,\infty)$, and there exists $C_1>0$ such that
\[
u(t,x)\leq C_1 \mbox{ for } t\in [0,T_0) \mbox{ and } x\in [g(t),
h(t)].
\]
Then there exists $C_2$ depending on $C_1$ but independent of $T_0$
such that
\[
-g'(t), h'(t) \in (0, C_2]\; \mbox{ for } t\in (0, T_0).
\]
Moreover, the solution can be extended to some  interval $(0, T)$
with $T>T_0$.
\end{lem}
\begin{proof}
Since $f$ is $C^1$ and $f(0)=0$, there exists $K>0$ depending on
$C_1$ such that $f(u)\leq K$ for $u\in [0, C_1]$. We may then follow
the proof of Lemma 2.2 of \cite{DuLin} to construct an upper
solution of the form
\[
w(t,x)=C_1 \big[2M(h(t)-x)-M^2 (h(t)-x)^2\big]
\]
for some suitable $M>0$,  over the region
\[\{(t,x):0<t<T_0, h(t)-M^{-1}<x<h(t)\}
\]
to prove that $h'(t)\leq C_2$ for $t\in (0, T_0)$. The proof for
$g'(t)\geq -C_2$ is parallel.

Thus,  for $t\in [0, T_0)$,
\[
  -g(t), h(t)\in [h_0, h_0+C_2t],\quad -g'(t), h'(t)\in (0, C_2].
\]
We now fix $\delta_0\in (0, T_0)$. By standard $L^p$ estimates, the
Sobolev embedding theorem, and the H\"{o}lder estimates for
parabolic equations, we can find $C_3>0$ depending only on
$\delta_0$, $T_0$, $C_1$, and $C_2$ such that
 $||u(t,\cdot)||_{C^{2}([g(t), h(t)])}\leq
C_3$ for $t\in [\delta_0, T_0)$. It then follows from the proof of
Theorem~\ref{thm:local} that there exists a $\tau>0$ depending on
$C_3$, $C_2$, and $C_1$ such that the solution of problem \eqref{p}
with initial time $T_0-\tau$ can be extended uniquely to the time
$T_0+\tau $.
 (This is similar to the proof of Theorem 2.3 in
\cite{DuLin}.)
\end{proof}

The above lemma implies that the solution of \eqref{p} can be
extended as long as $u$ remains bounded. In particular, the free
boundaries never blow up when $u$ stays bounded. We have the
following result.

\begin{thm}
\label{thm:global} Suppose that \eqref{cond1} holds. Then \eqref{p}
has a unique solution defined on some maximal interval $(0, T^*)$
with $T^*\in (0,\infty]$. Moreover, when $T^*<\infty$, we have
\[
\lim_{t\to T^*} \max_{x\in [g(t), h(t)]}u(t,x)=\infty. \]
 If we further assume that \eqref{cond2} holds, then $T^*=\infty$.
\end{thm}

\begin{proof}
We only need to show that $T^*=\infty$ if \eqref{cond2} holds; the
other conclusions follow directly from Theorem \ref{thm:local} and
Lemma \ref{lem:bound-general}.

Comparing $u(t,x)$ with the solution of the ODE
\[
v_t=f(v),\; v(0)=\|u_0\|_\infty,
\]
we obtain $u(t,x)\leq v(t)\leq \|u_0\|_\infty e^{Kt}$, since
$f(v)\leq Kv$. In view of Lemma \ref{lem:bound-general}, we must
have $T^*=\infty$. \qquad\end{proof}

\bigskip

The rest of this section is devoted to the proof of Theorem
\ref{thm:convergence}. We need a lemma first.

\begin{lem}
\label{lem:center} Suppose that $(u(t,x),g(t), h(t))$ is a solution
of \eqref{p} as given in Theorem \ref{thm:convergence}. Then
\begin{equation}
\label{g+h} -2h_0<g(t)+h(t)<2h_0 \; \mbox{ for all } t>0,
\end{equation}
\begin{equation}
\label{rough-symmetry}
 u_x(t,x)>0>u_x(t,y) \mbox{ for all $t>0$,
$x\in [g(t), -h_0]$ and $y\in [h_0, h(t)]$.}
\end{equation}
\end{lem}
\begin{proof}
By continuity, $g(t)+h(t)>-2h_0$ for all small $t>0$. Define
\[
T:=\sup\{s: g(t)+h(t)>-2h_0 \mbox{ for all } t\in (0,s)\}.
\]
We show that $T=+\infty$. Otherwise $T$ is a positive number and
\[
g(t)+h(t)>-2h_0 \mbox{ for $t\in (0,T)$, }\;  g(T)+h(T)=-2h_0.
\]
Hence
\begin{equation}
\label{g'+h'}
 g'(T)+h'(T)\leq 0.
\end{equation}

We now derive a contradiction by considering
\[
w(t,x):=u(t,x)-u(t, -x-2h_0)
\]
over the region
\[G:=\{(t,x): t\in [0,T],\; g(t)\leq x\leq -h_0\}.
\]
Since $- h_0 \leq -x-2h_0\leq -g(t)-2h_0\leq h(t)$ when $(t,x)\in
G$, $w$ is well-defined over $G$ and it satisfies
\[
w_t=w_{xx}+c(t,x)w \mbox{ for } 0<t\leq T,\; g(t)<x<-h_0,
\]
with some $c\in L^\infty(G)$, and
\[w(t,-h_0)=0,\; w(t, g(t))<0 \mbox{ for } 0<t<T.
\]
Moreover,
\[
w(T,g(T))=u(T,g(T))-u(T, -g(T)-2h_0)=u(T,g(T))-u(T, h(T))=0.
\]
 Applying the strong maximum principle and the Hopf lemma, we deduce
\[
w(t,x)<0 \mbox{ for $0<t\leq T$ and $g(t)<x<-h_0$, and $w_x(T,
g(T))<0$}.
\]
But
\[ w_x(T,g(T))=u_x(T,g(T))+u_x(T, h(T))=-[g'(T)+h'(T)]/\mu.
\]
Thus we have
\[
g'(T)+h'(T)>0,
\]
a contradiction to \eqref{g'+h'}. This proves that $g(t)+h(t)>-2h_0$
for all $t>0$. We can similarly prove $g(t)+h(t)<2h_0$ by
considering
\[
v(t,x):=u(t,x)-u(t,2h_0-x) \mbox{ over } \{(t,x): t>0, h_0\leq x\leq
h(t)\}.
\]

With \eqref{g+h} proven, it is now easy to prove
\eqref{rough-symmetry}. For any fixed $\ell\in (g_\infty, -h_0]$, we
can find a unique $T\geq 0$ such that $g(T)=\ell$. We now consider
\[
z(t,x):=u(t,x)-u(t, 2\ell-x)
\]
over $G_{\ell}:=\{(t,x): t>T, g(t)<x<\ell\}$. We have
\[
z_t=z_{xx}+c(t,x)z \mbox{ in } G_{\ell},
\]
\[
z(t, g(t))<0 \mbox{ and }  z(t, \ell)=0 \mbox{ for } t>T.
\]
Hence we can apply the strong maximum principle and the Hopf lemma
to deduce
\[
z(t,x)<0 \mbox{ in } G_{\ell},\; z_x(t, \ell)>0 \mbox{ for } t>T.
\]
Since
\[z_x(t,\ell)=2u_x(t,\ell),
\]
we thus have
\[ u_x(t,g(T))>0 \mbox{ for } t>T.
\]
Now for any $t>0$ and $x\in (g(t),-h_0]$, we can find a unique $T\in
[0,t)$ such that $x=g(T)$. Hence $u_x(t, x)>0$. This inequality is
also true for $x=g(t)$, which is a consequence of the Hopf lemma
applied directly to \eqref{p}.

The proof for the other inequality in \eqref{rough-symmetry} is
similar.
\end{proof}

\smallskip

\noindent
 {\bf Proof of Theorem \ref{thm:convergence}:}\  We will
make use of Lemma \ref{lem:center} and then follow the ideas of
\cite{DM} with suitable variations.

Let $(u,g,h)$ be as given in Theorem \ref{thm:convergence}. Then in
view of Lemma \ref{lem:center}, $I_\infty:=(g_\infty, h_\infty)$ is
either a finite interval or $\R^1$. Denote by $\omega(u)$ the
$\omega$-limit set of $u(t,\cdot)$ in the topology of
$L^\infty_{loc}(I_\infty)$. Thus a function $w(x)$ belongs to
$\omega(u)$ if and only if there exists a sequence
$0<t_1<t_2<t_3<\cdots\to\infty$ such that
\begin{equation}\label{def-omega}
\lim_{n\to\infty} u(t_n,x)=w(x) \quad\ \ \hbox{locally uniformly
in}\ \ I_\infty.
\end{equation}

By local parabolic estimates, we see that the convergence
\eqref{def-omega} implies convergence in the $C^2_{loc}(I_\infty)$
topology. Thus the definition of $\omega(u)$ remains unchanged if
the topology of $L^\infty_{loc}(I_\infty)$ is replaced by that of
$C^2_{loc}(I_\infty)$.

It is well-known that $\omega(u)$ is compact and connected, and it
is an invariant set.  This means that for any $w\in\omega(u)$ there
exists an entire orbit (namely a solution of $W_t=W_{xx}+f(W)$
defined for all $t\in\R^1$ and $x\in I_\infty$) passing through $w$.
Choosing a suitable sequence $0<t_1<t_2<t_3<\cdots\to\infty$, we can
find such an entire solution $W(t,x)$ with $W(0,x)=w(x)$ as follows:
\begin{equation}
\label{u-to-W}
 u(t+t_n,x)\to W(t,x)\quad\hbox{as}\ \ n\to\infty.
\end{equation}
Here the convergence is understood in the $L^\infty_{loc}$ sense in
$(t,x)\in\R^1\times I_\infty$, but, by parabolic regularity, it
takes place in the $C^{1,2}_{loc}(\R^1\times I_\infty)$ sense.

\smallskip

For clarity we divide the arguments below into four parts, each
proving a specific claim.

\smallskip \noindent
{\bf Claim 1:}  $\omega(u)$ consists of solutions of
\begin{equation}
\label{stationary} v_{xx}+f(v)=0,\quad x\in I_\infty.
\end{equation}

Let $w(x)$ be an arbitrary element of $\omega(u)$ and $W(t,x)$  the
entire orbit satisfying $W(0,x)=w(x)$. Since $W$ is a nonnegative
solution of
\[
W_t=W_{xx}+f(W),\quad t\in\R^1, \; x\in I_\infty,
\]
and $f(0)=0$, by the strong maximum principle we have either
$W(t,x)>0$ for all $t\in\R^1$ and $x\in I_\infty$, or $W\equiv 0$. (Note that if $I_\infty$ is a finite interval, then
it can be shown that $W(t, g_\infty)=W(t, h_\infty)=0$ for all $t\in\R^1$.)
In the latter case we have $w\equiv 0$, which is a solution to
\eqref{stationary}. In what follows we assume the former, namely
$w>0$.

By Lemma \ref{lem:center}, we see that $w'(x)\geq 0$ for $x\in
(g_\infty, -h_0]$ and $w'(x)\leq 0$ for $x\in [h_0, h_\infty)$. Thus
there exists $x_0\in (-h_0,h_0)$ such that $w'(x_0)=0$,
$w(x_0)=\|w\|_\infty>0$.

Let $v(x)$ be the solution of the following initial value problem:
\[
v''+f(v)=0,\quad\ \ v(x_0)=w(x_0),\ \ \ v'(x_0)=0.
\]
Then $v$ is symmetric about $x=x_0$.  Since $w(x_0)>0$, $v$ is
either a positive solution of \eqref{stationary} in $\R^1$ or a
solution of \eqref{stationary} with compact positive support, namely
there exists $R_0>0$ such that
\[
v(x)>0 \mbox{ in } (x_0-R_0, x_0+R_0),
\; v(x_0\pm R_0)=0 \mbox{ or } v(x_0\pm R_0)=\infty.
\]
We may now follow the argument in the proof of Lemma 3.4 in
\cite{DM} (with obvious minor variations) to conclude that $w\equiv
v$. This proves Claim 1.

\smallskip \noindent
{\bf Claim 2:} If $I_\infty$ is a finite interval, then
$\omega(u)=\{0\}$.

Otherwise by Claim 1, $\omega(u)$ contains a nontrivial nonnegative
solution $v$ of the problem
\[
v_{xx}+f(v)=0 \mbox{ in } I_\infty,\; v(g_\infty)=v(h_\infty)=0.
\]
Due to $f(0)=0$, by the strong maximum principle and the Hopf lemma,
we have $v>0$ in $I_\infty$ and $v'(g_\infty)>0>v'(h_\infty)$. By
definition, along a sequence $t_n\to+\infty$, $u(t_n,x)\to v(x)$ in
$C^1_{loc}(I_\infty)$. We claim that there exists $\alpha>0$ so that, by passing to a subsequence,
 $\|u(t_n,\cdot)-v(\cdot)\|_{C^{1+\alpha}([g(t_n), h(t_n)])}\to 0$
as $n\to\infty$. Indeed, if we make a change of the variable $x$ to
reduce $[g(t), h(t)]$ to the fixed finite interval $[-h_0,h_0]$ as
in the proof of Theorem 2.1 of \cite{DuLin}, so that the solution
$u(t,x)$ is changed to $\tilde u(t,x)$, and $v(x)$ is changed to
$\tilde v(x)$. Then we can apply the $L^p$ estimates (and Sobolev embeddings)
on the reduced equation with Dirichlet boundary conditions to
conclude that $\tilde u(t+\cdot, \cdot)$ has a common bound in
$C^{\frac{1+\nu}{2},1+\nu}([0,1]\times[-h_0,h_0])$ for all $t\geq 1$, say
\begin{equation}
\label{bound-tilde-u} \|\tilde
u(t+\cdot,\cdot)\|_{C^{\frac{1+\nu}{2},1+\nu}([0,1]\times[-h_0,h_0])}\leq C_0 \;\;\forall t\geq 1.
\end{equation}
  Hence by extraction of a subsequence we may
assume that $\tilde u(t_n,x)\to V(x)$ in $C^{1+\frac{\nu}{2}}([-h_0,h_0])$. But from
$u(t_n,x)\to v(x)$ we know that $\tilde u(t_n,x)\to \tilde v(x)$.
Thus we necessarily have $V(x)\equiv \tilde v(x)$, and thus
$\|u(t_n,\cdot)-v(\cdot)\|_{C^{1+\frac{\nu}{2}}([g(t_n), h(t_n)])}\to 0$.

It follows
that
\[
h'(t_n)=-\mu u_x(t_n,h(t_n))\to -\mu v'(h_\infty)>0 \mbox{ as }
n\to\infty.
\]
Hence for all large $n$, say $n\geq n_0$,
\[
h'(t_n)\geq \delta:=-\mu v'(h_\infty)/2>0.
\]
On the other hand, from \eqref{bound-tilde-u}, we also deduce that
\[
\| u(t+\cdot,\cdot)\|_{C^{\frac{1+\nu}{2},1+\nu}(Q_t])}\leq C_1 \;\;\forall t\geq
1,\]
 with
 $ Q_t:=\{(s,x): s\in [0,1], g(t+s)\leq x\leq h(t+s)\}$.
It follows that
$ h'(t)=-\mu u_x(t,h(t))
$
is uniformly continuous in $t$ for $t\geq 1$.
Therefore $h'(t)\geq \delta/2$ for $t\in [t_n, t_n+\epsilon]$ and
$n\geq n_0$ for some $\epsilon>0$ sufficiently small but independent
of $n$ (we may assume without loss of generality that
$t_{n+1}-t_n\geq 1$ for all $n$). Since $h'(t)>0$ for all $t>0$, we
thus have
\[h_\infty\geq h_0+\Sigma_{n=n_0}^\infty
\int_{t_n}^{t_n+\epsilon}h'(t)dt=+\infty,
\]
 a
contradiction to the assumption that $I_\infty$ is a finite
interval. The proof of Claim 2 is now complete.

\smallskip \noindent
{\bf Claim 3:} If $I_\infty=\R^1$, then $\omega(u)$ is either a
constant or $\omega(u)=\{v(\cdot+\alpha): \alpha\in
[\alpha_1,\alpha_2]\}$ for some interval $
[\alpha_1,\alpha_2]\subset [-h_0,h_0]$, where $v$ is an evenly
decreasing positive solution of \eqref{stationary}.

In view of Lemma \ref{lem:center}, we only need to consider the case
that $I_\infty=\R^1$ and $\omega(u)$ is not a singleton. Then since
$\omega(u)$ is connected and compact in the topology of $C^2_{\rm
loc}(\R^1)$, and every function $w(x)$ in $\omega(u)$ achieves its
maximum at some $x_0\in [-h_0, h_0]$, we find that there exist
$0\leq \gamma^-\leq\gamma^+$ such that $\omega(u)$ consists of
solutions $v_{\alpha,\beta}\; (\beta\in [\gamma^-,\gamma^+],
\alpha\in [\alpha_1^\beta,\alpha_2^\beta])$ of \eqref{stationary}
satisfying
\[
v_{\alpha,\beta}(x)=v_\beta(x+\alpha) \; \; (\alpha\in
[\alpha_1^\beta,\alpha_2^\beta]),
\]
\[
\|v_\beta\|_\infty=v_\beta(0)=\beta,\; v_\beta'(0)=0,\;
[\alpha_1^\beta, \alpha_2^\beta]\subset [-h_0, h_0]\;\;
(\beta\in[\gamma^-,\gamma^+]).
\]
Thus each $v_{\alpha,\beta}$ is either a constant or a symmetrically
decreasing solution of \eqref{stationary}. If $\gamma^-<\gamma^+$,
then we may use $v_\beta$ $(\beta\in [\gamma^-,\gamma^+])$ to deduce
a contradiction in the same way as in section 3.3 of \cite{DM}. Thus
$\gamma^-=\gamma^+$. Let $V_0(x)$ be the unique solution of
\eqref{stationary} satisfying
\[
V(0)=\gamma^-,\; V'(0)=0.
\]
If $V_0$ is a constant, then clearly $\omega(u)=\{V_0\}$. Otherwise
$V_0$ is an evenly decreasing positive solution of
\eqref{stationary}, and $\omega(u)=\{V_0(\cdot+\alpha): \alpha\in
[\alpha_1,\alpha_2]\}$, $[\alpha_1,\alpha_2]\subset [-h_0,h_0]$.

\smallskip
\noindent {\bf Claim 4:} If $\omega(u)=\{V_0(\cdot+\alpha):
\alpha\in [\alpha_1,\alpha_2]\}$ for some interval
$[\alpha_1,\alpha_2]\subset [-h_0,h_0]$, then there exists a
continuous function $\gamma: [0,\infty)\to [-h_0,h_0]$ such that
\[
u(t,x)-V_0(x+\gamma(t))\to 0 \mbox{ as } t\to\infty \mbox{ locally
uniformly in } \R^1.
\]

Write $w(t,x)=u_x(t,x)$. Then
\[
w_t=w_{xx}+f'(u(t,x))w \mbox{ for } t>0,\; x\in (g(t), h(t)),
\]
and $w(t,g(t))>0$, $w(t, h(t))<0$ for all $t>0$. Therefore by the
zero number result of \cite{A}, for all large $t$, say $t\geq T$,
$w(t,x)$ has a fixed finite number of zeros, all nondegenerate.
Denote them by
\[
x_1(t)<x_2(t)<... <x_m(t)\; (m\geq 1).
\]
Then each $x_i(t)$ is a continuous function of $t$. Due to Lemma
\ref{lem:center}, we must have $-h_0\leq x_i(t)\leq h_0$ for
$i=1,..., m$ and $t\geq T$. We show that $m=1$. For fixed $\alpha\in
[\alpha_1,\alpha_2]\subset [-h_0,h_0]$, since $V_0(\cdot+\alpha)\in
\omega(u)$, there exists $t_n\to\infty$ such that $u(t_n,x)\to
V_0(x+\alpha)$ in $C_{loc}^2(\R^1)$. Since $V_0'(x+\alpha)$ has a
unique nondegenerate zero $x=-\alpha\in [-h_0,h_0]$, we find that
for all large $n$, $w(t_n,x)=u_x(t_n,x)$ has in $[-2h_0,2h_0]$ a
unique nondegenerate zero $\alpha_n$ near $-\alpha$. By Lemma
\ref{lem:center}, we necessarily have $\alpha_n\in (-h_0,h_0)$. On
the other hand, we know that $x_1(t_n),..., x_m(t_n)$ are all the
zeros of $w(t_n,x)$ in $[-h_0,h_0]$. Thus we must have $m=1$ and
$x_1(t_n)=\alpha_n$. This proves $m=1$.

Define $\gamma(t)=-x_1(t)$ for $t\geq T$, and extend $\gamma(t)$ to
a continuous function for $t\in [0,T]$ such that $\gamma(t)\in
[-h_0,h_0]$ for all $t$. We prove that
\[
u(t,x)-V_0(x+\gamma(t))\to 0 \mbox{ as } t\to\infty \mbox{ locally
uniformly in } \R^1.
\]
Otherwise we can find $t_n\to\infty$, a bounded sequence
$\{x_n\}\subset\R^1$ and some $\epsilon_0>0$ such that for all
$n\geq 1$,
\[
|u(t_n,x_n)-V_0(x_n+\gamma(t_n))|\geq \epsilon_0.
\]

By passing to a subsequence of $t_n$, still denoted by itself, we
may assume $u(t_n,\cdot)\to V_0(\cdot+\alpha)$ in $C_{loc}^2(\R^1)$
for some $\alpha\in [\alpha_1,\alpha_2]$. Hence $w(t_n,\cdot)\to
V_0'(\cdot+\alpha)$ in $C_{loc}^1(\R^1)$. This implies that
$\gamma(t_n)=-x_1(t_n)\to \alpha$, and thus, due to the boundedness
of $\{x_n\}$, we have
\[
V_0(x_n+\alpha)-V_0(x_n+\gamma(t_n))\to 0 \mbox{ as } n\to\infty.
\]
It follows that
\begin{eqnarray*}
\epsilon_0&\leq& |u(t_n,x_n)-V_0(x_n+\gamma(t_n))|
\\
&\leq&
|u(t_n,x_n)-V_0(x_n+\alpha)|+|V_0(x_n+\alpha)-V_0(x_n+\gamma(t_n))|\to
0
\end{eqnarray*}
as $n\to\infty$. This contradiction proves our claim.

 The proof of Theorem
\ref{thm:convergence} is now complete. {\hfill $\Box$}

\section{Conditions for vanishing}\label{sec:vanishing}
In this section we prove some sufficient conditions that imply
vanishing ($u\rightarrow 0$). The following upper bound is an easy
consequence of the standard comparison principle.

\begin{lem}\label{heat>}
Assume that $f$ satisfies \eqref{cond1} and \eqref{cond2}. Then, for
any $h_0 >0$ and any $\phi\in \mathscr{X}(h_0)$,
\begin{equation}\label{u < heat}
u(t,x;\phi) \leq \frac{e^{Kt}}{2\sqrt{\pi t}} \int_{-h_0}^{h_0}
\phi(x) dx\quad \mbox{for }\; g(t)\leq x\leq h(t),\ t>0.
\end{equation}
\end{lem}
\begin{proof} Consider the Cauchy problem
\begin{equation}\label{prob-w}
\left\{
\begin{array}{ll}
 w_t =w_{xx} +Kw, & x\in\R^1,\ t>0,\\
 w(0,x)= \Phi (x), &  x \in \R^1,
  \end{array}
 \right.
 \end{equation}
where
$$
\Phi (x) = \left\{
\begin{array}{ll}
     \phi(x),\ \ & x\in (-h_0, h_0),\\
      0, \ \ & x\not\in (-h_0, h_0).
\end{array}
\right.
$$
 Then from the expression of $w$ by
the fundamental solution we obtain
$$
{w} (t,x)  =  \frac{e^{Kt}}{\sqrt{4\pi t}} \int_{\mathbb{R}}
e^{-\frac{(x-\xi)^2}{4t}} {w} (0,\xi)d\xi  \leq
\frac{e^{Kt}}{2\sqrt{\pi t}} \int_{-h_0}^{h_0} \phi(\xi)d\xi.
$$

By the standard comparison theorem, we have $u(t,x;\phi) \leq
{w}(t,x)$ for $t>0$ and $x\in [g(t), h(t)]$, and the required
inequality follows. \end{proof}

\begin{thm}\label{thm:vanishing}
Let $h_0 >0$ and $\phi \in \mathscr{X}(h_0)$. Then
$I_\infty:=(g_\infty, h_\infty)$ is a finite interval and
$\lim_{t\to\infty} \|u(t, \cdot; \phi)\|_{L^\infty([g(t), h(t)])}=0$
if one of the following conditions holds:

\begin{itemize}
\item [\rm (i)] $f$ is of {\rm (f$_M$)} type,
$h_0 < \pi/(2\sqrt{f'(0)})$ and $\|\phi\|_{L^\infty}$ is
sufficiently  small;

\item [\rm (ii)] $f$ is of {\rm (f$_B$)} or
{\rm (f$_C$)} type, and $\|\phi \|_{L^\infty} \leq \theta $;

\item [\rm (iii)] $f$ is of {\rm (f$_B$)} or
{\rm (f$_C$)} type, and for $K$ in \eqref{cond2},
\begin{equation}\label{small int}
\int_{-h_0}^{h_0} \phi(x) dx \leq \theta \cdot \sqrt{
\frac{2\pi}{eK}}.
\end{equation}
\end{itemize}
\end{thm}
\begin{proof} (i) \  Since $h_0 < \pi/(2\sqrt{f'(0)})$, there
exists a small $\delta >0$ such that
\begin{equation}\label{choice of delta}
\frac{\pi^2}{4 (1+\delta)^2 h^2_0} - f'(0) \geq 2 \delta.
\end{equation}
Moreover, there exists an $s
>0$ small such that
$$
\pi \mu s \leq \delta^2 h^2_0, \qquad f(u) \leq (f'(0) +\delta) u
\quad \mbox{for } u\in [0,s].
$$
Set
$$
k(t) := h_0 \Big( 1+\delta - \frac{\delta}{2} e^{-\delta t}
\Big)\quad \mbox{and} \quad w(t,x):= s e^{-\delta t} \cos\Big(
\frac{\pi x}{2 k (t)}\Big).
$$
Clearly $w(t, -k(t))=w(t, k(t))=0$. A direct calculation shows that,
for $t>0$ and $x\in [-k(t), k(t)]$,
$$
w_t - w_{xx} -f(w) \geq   \left( -\delta + \frac{\pi^2}{4k^2(t)}
-f'(0) -\delta  \right) w  \geq  0.
$$
On the other hand, by the choice of $s$ we have
$$
\mu w_x(t, -k(t))=- \mu w_x(t, k(t)) = \frac{\pi \mu s}{2k(t)}
e^{-\delta t} \leq \frac{\pi \mu s}{2h_0 } e^{-\delta t} \leq
\frac{\delta^2 h_0}{2} e^{-\delta t} = k'(t).
$$
Therefore, $(w(t,x), -k(t), k(t))$ will be an upper solution of
\eqref{p} if $w(0,x)\geq \phi (x)$ in $[-h_0,h_0]$.

Choose $\sigma_1 := s\cos \frac{\pi}{2+\delta}$, which depends only
on $\mu, h_0$ and $f$. Then when $\|\phi\|_{L^\infty} \leq \sigma_1$
we have $\phi (x) \leq \sigma_1 \leq w(0,x)$ in $[-h_0, h_0]$, since
$h_0 < k(0)= h_0 (1+\frac{\delta}{2})$. By Lemma \ref{lem:comp1} we
have
$$
h(t) \leq k(t) \leq h_0 (1+\delta),\; h_\infty<\infty.
$$
Hence $I_\infty$ is a finite interval and by Theorem
\ref{thm:convergence}, $u\to 0$ as $t\to\infty$ locally uniformly in
$I_\infty$. In view of Lemma \ref{lem:center}, this implies that
$\lim_{t\to\infty} \|u(t,\cdot)\|_{L^\infty([g(t), h(t)])}=0$.

\bigskip

 (ii) {(The (f$_B$) case)}\  Since $u\equiv \theta$
is a stationary solution, by the strong comparison principle, there
exist $\eta_1 \in (0,\theta)$ and $t_1
>0$ such that
$$
u(t_1,x;\phi) \leq \eta_1 \quad \mbox{for }  x\in
[g_1,h_1]:=[g(t_1), h(t_1)].
$$
Since $f$ is of (f$_B$) type, there exists $M = M(\eta_1)
>0$ such that
$$
f(u) \leq -M u \quad \mbox{for } 0\leq u\leq \eta_1.
$$
It follows that $u(t,x;\phi) \leq \eta(t) := \eta_1 e^{-M(t-t_1)}$
for $t\geq t_1$. Choose $\rho > h_1$ such that $ 2M\rho^2 > \pi \mu
\eta_1 e^{Mt_1}$, and then choose $0<\delta <\min \{\frac{\rho}{2},
h_1\}$ small such that
\begin{equation}\label{small eta0}
u(t_1, x) < \frac{\sqrt{2}}{2} \eta_1 \quad \mbox{for } x\in [g_1,
g_1+\delta]\cup [h_1 -\delta, h_1].
\end{equation}

For $t\geq t_1$ we define
$$
\sigma(t):= \rho(2-e^{-Mt}) \quad \mbox{and} \quad k(t):= h_1 -
\delta +\sigma(t),
$$
(so $\rho \leq \sigma(t) \leq 2\rho$,  $k(t_1)>h_1$), and
$$
w(t,x):= \eta(t) \cos\left[ \frac{\pi(x-h_1
+\delta)}{2\sigma(t)}\right]\quad \mbox{for } h_1 - \delta\leq x\leq
k(t),\ t\geq t_1.
$$
Then, for $h_1 - \delta\leq x\leq k(t),\ t\geq t_1$, we have
$$
w_t-w_{xx} +Mw  =  \frac{\pi^2 w}{4\sigma^2(t)} + \eta (t) \sin
\left[ \frac{\pi(x-h_1 +\delta)}{2\sigma(t)}\right] \cdot \frac{\pi
(x-h_1 +\delta) \sigma'(t)}{2\sigma^2 (t)}
>0.
$$
$$
w(t, k(t))=0 \mbox{ and } -\mu w_x(t,k(t)) \leq  \frac{\pi \mu \eta
(t) }{2\rho} \leq M\rho e^{-Mt} = k'(t)
$$
by the choice of $\rho$. Moreover, $w(t, h_1-\delta)=\eta(t)\geq
u(t, h_1-\delta)$, and by \eqref{small eta0},
$$
u(t_1, x) < \frac{\sqrt{2}}{2} \eta_1 \leq w(t_1, x) \quad \mbox{for
}h_1 -\delta \leq x\leq h_1.
$$
Hence $(w(t,x),h_1-\delta, k(t))$ is an upper solution of \eqref{p}
for $t>t_1$ in the sense of Lemma \ref{lem:comp2}. By the conclusion
of this lemma  we have $h(t)\leq k(t)$, and hence
$$
h_\infty \leq \lim\limits_{t\rightarrow \infty} k(t)= h_1 -\delta +2
\rho <\infty.
$$
The rest of the proof is the same as in (i).

\bigskip

(ii) (The (f$_C$)  case)\  In this case $u\equiv \theta$ is again a
stationary solution, and by the standard comparison principle we
have $u(t,x;\phi) \leq \theta$ for all $t\geq 0$. Therefore, the
equation we are dealing with reduces to the heat equation $u_t =
u_{xx}$. As in the proof of Lemma \ref{heat>} we have
$$
u(t,x;\phi) \leq \frac{1}{2\sqrt{\pi t}} \int_{-h_0}^{h_0} \phi(\xi)
d\xi \leq  \frac{\theta h_0}{\sqrt{\pi t}} \quad \mbox{for }
g(t)\leq x\leq h(t),\ t>0.
$$
Therefore, we can find a large  $t_2 >0$ such that
\begin{equation}\label{C-def-eta}
\max_{g(t_2)\leq x\leq h(t_2)} u(t_2,x;\phi) \leq \eta_2 :=
\frac{1}{2}\cdot \min\Big\{\theta, \frac{\pi}{8\mu}\Big\}.
\end{equation}
Take $h_2 > \max\{-g(t_2), h(t_2)\}$ such that
\begin{equation}\label{C-ini-w}
u(t_2, x; \phi) < 2 \eta_2 \cos\Big( \frac{\pi x}{2h_2} \Big) \mbox{
for } x\in [g(t_2), h(t_2)].
\end{equation}
For this $h_2$ we set $\omega:= \pi/(4h_2)$ and define, for $t\geq
0$,
$$
k(t) := h_2 (2 -  e^{-\omega^2 t}),\quad w(t,x) := 2\eta_2 \cos\Big(
\frac{\pi x}{2 k(t)} \Big) e^{-\omega^2 t}.
$$
Then, $ h_2 \leq k(t) \leq 2h_2$, and for $t\geq 0$ and $-k(t)\leq x
\leq k(t)$ we have
$$
w_t - w_{xx} \geq \left( \frac{\pi^2 }{4 [k(t)]^2} -\omega^2
 \right) w \geq 0
$$
and, by the choice of $\eta_2$,
$$
k'(t)-\mu w_x(t,-k(t))=k'(t) +\mu w_x (t,k(t)) = e^{-\omega^2
t}\left[ \frac{\pi^2 }{16 h_2}  -\frac{\pi \mu \eta_2}{k(t)} \right]
\geq 0.
$$
 Hence $(w(t,x), -k(t), k (t))$ is an upper solution of \eqref{p} for
$t>t_2$. It follows that $h(t +t_2) \leq k(t)<2h_2$ for $t\geq 0$.
This implies that $h_\infty<\infty$ and the rest is as before.

\bigskip

(iii) By \eqref{u < heat}, we have
$$
u \Big( \frac{1}{2K}, x;\phi\Big) \leq \sqrt{\frac{eK}{2\pi}}
\int_{-h_0}^{h_0} \phi (x) dx \leq \theta\quad \mbox{for} \ \
g\Big( \frac{1}{2K}\Big) \leq x \leq h\Big( \frac{1}{2K}\Big).
$$
Then the conclusion follows from (ii). This proves the theorem.
\end{proof}

From  Theorem \ref{thm:vanishing} (ii), we immediately obtain

\begin{cor}
\label{cor:BC-vanishing-finite} If $f$ is of  {\rm (f$_B$)} or of
{\rm (f$_C$)} type, then
\[
\lim_{t\to\infty} \|u(t,\cdot)\|_{L^\infty([g(t), h(t)])}=0
\]
implies that $(g_\infty, h_\infty)$ is a finite interval.
\end{cor}

If $f$ is of (f$_M$) type, this conclusion is also true. In fact, a
much stronger version holds, namely
\[
h_\infty-g_\infty\leq \pi/\sqrt{f'(0)}.
\]
 This will  follow from Theorem \ref{thm:vanishing} (i) and
 Corollary \ref{cor:mono-spreading} in the next section; see
Corollary \ref{cor:mono-vanishing}.

\section{Waves of finite length and conditions for spreading}
In this section,  $f$ is always assumed to be of (f$_M$), or
(f$_B$), or (f$_C$) type.  In order to obtain sufficient conditions
guaranteeing spreading ($u\to 1$), we will construct suitable lower
solutions to \eqref{p} through ``waves of finite length'', obtained
by a phase plane analysis of the equation
\[
q''-cq'+f(q)=0.
\]

\subsection{Waves of finite length}\label{subsec:w-f}

For $Z\in (0,\infty)$, we look for a  pair  $(c,q(z))$ satisfying
\begin{equation}\label{prob-q}
\left\{
\begin{array}{l}
q'' -cq' +f(q)=0, \quad z\in [0,Z],\\
q(0)=0, \ \ q'(Z)=0,\ \ q(z)>0 \ \mbox{ in } (0,Z].
\end{array}
\right.
\end{equation}
We call such a $q(z)$ a ``wave of length $Z$ with speed $c$'', since
$w(t,x):=q(ct-x)$ satisfies
\[
\left\{
\begin{array}{l}
w_t=w_{xx}+f(w) \mbox{ for } t\in\R^1, \; x\in (ct-Z, ct),\\
w_x(t, ct-Z)=0,\; w(t, ct)=0.
\end{array}
\right.
\]
Such $w$ will be used to construct lower solutions to \eqref{p}. We
will mainly consider waves of speed $c=0$ (stationary waves) and of
speed $c>0$ small.

 Using $q'$ to denote $\frac{dq}{dz}$, we can rewrite the first equation
 in \eqref{prob-q} into the equivalent
form
\begin{equation}\label{eq-qp}
\left\{
\begin{array}{l}
q'=p,\\
p'=cp -f(q),
\end{array}
\right.
\end{equation}
or,
\begin{equation}\label{pq}
\frac{dp}{dq} = c- \frac{f(q)}{p}\quad \mbox{when } p\not= 0.
\end{equation}
For each $c\geq 0$ and $\omega >0$, we use $p^c (q;\omega)$ to
denote the unique solution of \eqref{pq} with initial condition
$p(q)|_{q=0} =\omega$. Such a solution is well-defined as long as it
stays positive.

In the case $c=0$, the positive solution of \eqref{pq} with
$p(q)|_{q=0} =\omega$ is given explicitly by
\begin{equation}\label{p0-left}
p^0(q;\omega) = \sqrt{\omega^2 - 2\int_0^q f(s)ds} \quad \mbox{for }
q\in [0,q^{\omega}),
\end{equation}
where $q^{\omega}$ is given by
$$
\omega^2 = 2\int_0^{q^{\omega}} f(s) ds.
$$
It follows that $q^{\omega} < 1$ if and only if $0 <\omega
<\omega^0$, where
$$
\omega^0 := \sqrt{2\int_0^1 f(s)ds}.
$$
Moreover, it is easily seen that $q^{\omega}$ is strictly increasing
in $\omega \in (0,\omega^0)$, and as $\omega \searrow 0$,
$q^{\omega} \searrow 0$ in the (f$_M$) case, $q^\omega \searrow
\bar{\theta} \in (\theta, 1)$ in the (f$_B$) case, where
$\bar{\theta} \in (\theta, 1)$ is determined by
$\int_0^{\bar{\theta}} f(s) ds =0$, and $q^{\omega} \searrow \theta$
in the (f$_C$) case.

The positive solution $p^0(q;\omega)$ ($q\in [0, q^\omega)$)
corresponds to a trajectory $(q_0(z),p_0(z))$ of \eqref{eq-qp} (with
$c=0$) that connects $(0,\omega)$ and $(q^\omega, 0)$ in the
$qp$-plane. We may assume that it passes through $(0,\omega)$ at
$z=0$ and approaches $(q^\omega,0)$ as $z$ goes to $z^\omega\in
(0,+\infty]$. Then using \eqref{eq-qp} with $c=0$ and
\eqref{p0-left} we easily deduce
\begin{equation}\label{q-eq}
z = \int_0^{q_0(z)} \frac{dr}{\sqrt{\omega^2 - 2\int_0^r f(s) ds}} =
\int_0^{q_0(z)} \frac{dr}{\sqrt{2 \int_r^{q^{\omega}} f(s) ds}}.
\end{equation}
Therefore
$$
z^\omega = \int_0^{q^{\omega}} \frac{dr}{\sqrt{\omega^2 - 2\int_0^r
f(s) ds}} = \int_0^{q^{\omega}} \frac{dr}{\sqrt{2
\int_r^{q^{\omega}} f(s) ds}} <+\infty
$$
for $0<\omega<\omega^0$. We now introduce the function
\begin{equation}
\label{def-Z} Z(q)=\int_0^{q} \frac{dr}{\sqrt{2 \int_r^{q} f(s)
ds}}.
\end{equation}
   In the (f$_M$) case, define
\begin{equation}\label{def:RM}
Z'_M := \inf\limits_{0< \omega < \omega^0} z^\omega =
\inf\limits_{0<q<1} Z(q);
\end{equation}
in the (f$_B$) case, set
\begin{equation}\label{def:RB1}
Z_B:=\inf\limits_{0< \omega < \omega^0} z^\omega =
\inf\limits_{\bar{\theta} <q<1} Z(q);
\end{equation}
and in the (f$_C$) case, define
\begin{equation}\label{def:RC1}
Z_C:=\inf\limits_{0< \omega < \omega^0} z^\omega
=\inf\limits_{\theta <q<1} Z(q).
\end{equation}

In the (f$_M$) case, as $\omega \searrow 0$, we have
$q^{\omega}\searrow 0$ and so
$$
z^\omega = \int_0^{q^{\omega}} \frac{(1+o(1)) \
dr}{\sqrt{f'(0)}\sqrt{{(q^{\omega})^2 - r^2}} } = \frac{\pi
}{2\sqrt{f'(0)}} +o (1).
$$
 This implies that
\begin{equation}
\label{def-Z_M} Z'_M \leq Z_M:= \pi/ (2\sqrt{f'(0)}).
\end{equation}
It is easily seen that $Z'_M,\, Z_B$ and $Z_C$ are all positive.

\bigskip

As a first application of the above analysis, we have the following
result.

\begin{lem}
\label{lem:v_Z} If $f$ is of {\rm (f$_M$)} type and $Z>Z'_M$, or of
{\rm (f$_B$)} type and $Z\geq Z_B$, or of {\rm (f$_C$)} type and
$Z\geq Z_C$, then the elliptic boundary value problem
\begin{equation}
\label{v_Z} v_{xx}+f(v)=0 \mbox{ in } (-Z,Z),\; v(-Z)=v(Z)=0
\end{equation}
has at least one positive solution $v_Z$. Moreover, any positive
solution $v_Z$ of \eqref{v_Z} satisfies $\|v_Z\|_\infty<1$; in
addition, $\|v_Z\|_\infty>\overline\theta$ if $f$ is of {\rm
(f$_B$)} type, and $\|v_Z\|_\infty>\theta$ if $f$ is of {\rm
(f$_C$)} type.
\end{lem}
\begin{proof}
We only consider the case that $f$ is of (f$_B$) type; the proofs of
the other cases are similar.

Let $Z>Z_B$. Then from the definition of $Z_B$ we can find
$\omega_*\in (0, \omega_0)$ and correspondingly
$q_*:=q^{\omega_*}\in (\overline\theta, 1)$ such that
$z_*:=z^{\omega_*}=Z(q_*)\in (Z_B, Z)$. Let $(q(z), p(z))$ be the
trajectory of \eqref{eq-qp} (with $c=0$) that passes through $(0,
\omega_*)$ at $z=0$ and approaches $(q_*, 0)$ as $z$ goes to $z_*$.
Then $q(z)$ satisfies
\[
q''+f(q)=0 \mbox{ in } (0, z_*),\; q(0)=0,\; q'(z_*)=0.
\]
If we define
\[
\underline v(x):=\left\{\begin{array}{ll} q(x+z_*), & x\in [-z_*,
0],\\
q(-x+z_*), & x\in [0, z_*],\\
0, & x\in [-Z,Z]\backslash [-z_*,z_*].
\end{array}
\right.
\]
 Then one easily checks that $\underline v$ is a (weak) lower
solution of \eqref{v_Z}. Clearly any constant $C \geq 1$ is an upper
solution of \eqref{v_Z}. Therefore we can use the standard upper and
lower solution argument to conclude that \eqref{v_Z} has a maximal
positive solution $\hat v_Z$, and $\underline v(x)<\hat v_Z(x)<1$ in
$(-Z, Z)$.

We now prove that \eqref{v_Z} also has a positive solution for
$Z=Z_B$. Let $Z_n$ be a sequence decreasing to $Z_B$ and $v_n$ a
positive solution of \eqref{v_Z} with $Z=Z_n$. Setting $V_n:=v_n(Z_n
x)$ we find that $V_n$ is a positive solution of
\[
V''+Z_n^2f(V)=0 \mbox{ in } (-1,1),\; V(-1)=V(1)=0.
\]
Since $Z_n^2f(V_n)$ is a bounded sequence in $L^\infty([-1,1])$ it
follows from standard regularity theory that by passing to a
subsequence, $V_n\to V^*$ in $C^1([0,1])$ and $V^*$ is a weak (and
hence classical) nonnegative solution of
\[
V''+Z_B^2f(V)=0 \mbox{ in } (-1,1),\; V(-1)=V(1)=0.
\]
We claim that $V^*\not\equiv 0$. Arguing indirectly we assume that
$V^*\equiv 0$, and let $\hat V_n:=V_n/\|V_n\|_\infty$. Then
\[
\hat V_n''+c_n(x)\hat V_n=0 \mbox{ in } (-1,1),\; \hat V_n(-1)=\hat
V_n(1)=0,
\]
with $c_n=Z_n^2f(V_n)/V_n$ a bounded sequence in $L^\infty([-1,1])$.
As before, by standard elliptic regularity we have $\hat V_n\to\hat
V$ in $C^1([-1,1])$ subject to a subsequence. Moreover, since
$c_n\to Z_B^2f'(0)$, we deduce
\begin{equation}
\label{hat-V} \hat V''+Z_B^2f'(0)\hat V=0 \mbox{ in } (-1,1),\; \hat
V(-1)=\hat V(1)=0.
\end{equation}
Since $\|\hat V\|_\infty=1$ and $\hat V\geq 0$, by the strong
maximum principle we conclude that $\hat V$ must be a positive
solution of \eqref{hat-V}. This implies that $Z_B^2f'(0)$ is the
first eigenvalue of $(-\frac{d^2}{dx^2})$ over $(-1,1)$ with
Dirichlet boundary conditions, and hence must be positive. But this
is a contradiction to $f'(0)<0$. Thus $V^*\not\equiv 0$. By the
strong maximum principle we see that it is a positive solution of
\eqref{v_Z} with $Z=Z_B$.

Now let $v_Z$ be any positive solution of \eqref{v_Z}. Then clearly
$(q(z), p(z)):=(v_Z(Z-z), -v'_Z(Z-z))$ is a trajectory for
\eqref{eq-qp} (with $c=0$) passing throw $(0, \omega)$ at $z=0$ and
approaching $(q^\omega, 0)$ as $z$ goes to $Z$, where
$\omega:=-v_Z'(Z)$ and $q^\omega:=v_Z(0)<1$. Since $q^\omega$ is
strictly increasing and $q^\omega$ decreases to $\overline\theta$ as
$\omega$ decreases to 0, we find that $v_Z(0)>\overline\theta$.
\end{proof}

Next we consider \eqref{eq-qp} and \eqref{pq} for small $c>0$ as a
perturbation of the case $c=0$. It is easily seen that for small
$c>0$, \eqref{pq} with initial data $p^c(q)|_{q=0} =\omega \in (0,
\omega^0)$ has a solution $p^c (q; \omega)$  define on $q\in [0,
q^{c,\omega}]$ for some $q^{c,\omega}
>q^\omega$, and $p^c (q^{c,\omega};\omega) =0$. As before this
solution corresponds to a trajectory
$(q_c(z;\omega), p_c(z;\omega))$ that passes through $(0,\omega)$ at
$z=0$, and approaches $(q^{c,\omega}, 0)$ as $z$ goes to some
$z^{c,\omega}>0$.  Moreover, an elementary analysis yields the
following result.

\begin{lem}\label{lem:q-small-c}
For any fixed $\omega\in (0,\omega_0)$ and any small $\varepsilon
>0$, there exists $\delta>0$ small such that, if $c\in (0,\delta)$, then $q^{c,\omega} \in
(q^\omega, q^\omega +\varepsilon)$, and
$$
p^0(q;\omega) \leq p^c(q;\omega) \leq p^0 (q;\omega) +\varepsilon
\quad \mbox{for } q\in [0, q^\omega];
$$
moreover, $z^{c,\omega} \in(z^\omega -\varepsilon, z^\omega
+\varepsilon)$ and
$$
q_0 (z;\omega) \leq q_c (z;\omega ) \leq q_0(z;\omega) +\varepsilon
\quad \mbox{for } z\in [0, \min \{z^\omega, z^{c,\omega}\}].
$$
\end{lem}

Let us observe that $q(z):=q_c(z,\omega)$ is a solution of
\eqref{prob-q} with $Z=z^{c,\omega}$. Moreover, $q'(0)=\omega$. We
will use $q_c(z;\omega)$ below to construct lower solutions of
\eqref{p}.

\subsection{Conditions for spreading}

\begin{thm}
\label{thm:spreading} Suppose that the conditions in Lemma
\ref{lem:v_Z} are satisfied and $v_Z$ is a positive solution of
\eqref{v_Z}. If $(u,g,h)$ is a solution of \eqref{p} with $h_0\geq
Z$ and $u_0\geq v_Z$ in $[-Z,Z]$, then
\[
(g_\infty,h_\infty)=\R^1 \mbox{ and  $ \lim_{t\to\infty}
u(t,\cdot)=1 $ locally uniformly in $\R^1$.}
\]
\end{thm}
\begin{proof}
Since $v_Z$ is a stationary solution and $g(t)<-Z,\; h(t)>Z$ for
$t>0$, by the standard strong comparison principle we deduce
\[
u(t,x)>v_Z(x) \mbox{ in $[-Z,Z]$ for all $t>0$. }
\]
By Theorem \ref{thm:convergence}, $u(t,x)\to v(x)$ locally uniformly
in $(g_\infty, h_\infty)=\mathbb{R}^1$ as $t\to\infty$, where $v$ is
a nonnegative solution of \eqref{ellip}. $v$ must be a positive
solution since $v\geq v_Z$ in $[-Z,Z]$. Moreover, since
$[-Z,Z]\subset (g_\infty, h_\infty)$, we necessarily have $v>v_Z$ in
$[-Z,Z]$ due to the strong maximum principle.

Thus if we fix $t_0>0$ and extend $v_Z$ by 0 outside $[-Z,Z]$, then
we can find $\epsilon>0$ small such that for all $t\geq t_0$,
\[
u(t,x)>v_Z(x)+\epsilon \mbox{ in } [-h_0,h_0], \;
u(t,x)>v_Z(0)+\epsilon \mbox{ in  } [-\epsilon,\epsilon].
\]
As in the proof of Lemma \ref{lem:v_Z}, $v_Z$ corresponds to a
trajectory $(q_0(z;\omega), p_0(z;\omega))$ of \eqref{eq-qp} (with
$c=0$) that passes through $(0,\omega):=(0, -v_Z'(Z))$ at $z=0$ and
approaches $(q^\omega, 0):=(v_Z(0), 0)$ as $z$ goes to
$z^\omega:=Z$. By Lemma \ref{lem:q-small-c}, we can find $c>0$
sufficiently small such that the trajectory $(q_c(z;\omega),
p_c(z;\omega))$ of \eqref{eq-qp} that passes through $(0,\omega)$ at
$z=0$ and goes to $(q^{c,\omega},0)$ as $z$ approaches
$z^{c,\omega}$ satisfies
\[
z^{c,\omega}<z^\omega+\epsilon<h_0,\;
q^{c,\omega}<q^\omega+\epsilon<1,\]
\[
q_c(z^{c,\omega}-z;\omega)<q_0(z^{c,\omega}-z;\omega)+\epsilon/2<q_0(z^\omega-z;\omega)+\epsilon
\mbox{ in } [\epsilon,z^{c,\omega}].
\]
Since $p_c(z,\omega)=\frac{d}{dz} q_c(z;\omega)>0$ for $z\in
(0,z^{c,\omega})$, we find that
\[
\mbox{
$q_c(z^{c,\omega}-z;\omega)<q^{c,\omega}<q^\omega+\epsilon=v_Z(0)+\epsilon$
for $z\in (0,\epsilon]$.}
\]

Thus for such small $c>0$, we have
$$
u(t, x) > q_c (z^{c,\omega} -x; \omega) \quad \mbox{for } t\geq
t_0,\;  x\in [0, z^{c,\omega}].
$$
We now fix a small $c>0$ such that the above holds and
$c<\mu\omega$. Then define, for $t\geq 0$,
\[k(t):=z^{c,\omega} +ct
\]
and
\[ w(t,x):=\left\{
\begin{array}{ll}
q_c(k(t) -x;\omega), & x\in [ct, k(t)],\\
q_c(z^{c,\omega};\omega), &  x\in [0, ct].
\end{array}
\right.
\]
Since $q_c(z^{c,\omega};\omega)=q^{c,\omega}$ and
$f(q^{c,\omega})>0$, we find
\[
w_t\leq w_{xx}+f(w) \mbox{ for $t>0$ and $x\in (0, k(t))$.}
\]
Moreover,
\[
k(0)=z^{c,\omega}<h_0<h(t_0)
\]
and
\[
w(t, k(t))=0,\; k'(t)=c<\mu\omega=-\mu w_x(t, k(t)) \mbox{ for }
t>0.
\]
Thus we can apply the lower solution version of Lemma
\ref{lem:comp2} to conclude that
\[
h(t+t_0)\geq k(t) \mbox{ and } u(t+t_0,x)\geq w(t,x) \mbox{ for all
$t>0$ and $x\in [0,k(t)]$.}
\]
This implies that $h_\infty=\infty$ and the $\omega$-limit of $u$,
namely the positive solution $v(x)$ of \eqref{ellip}, is defined
over $\R^1$. Moreover, for $t>0$ and $x \in [0,ct]$, we have
$$
u(t+t_0 ,x) \geq w(t,x) = q^{c,\omega} > q^\omega.
$$
Hence $v(x)\equiv 1$ for $x\in \mathbb{R}^1$.
\end{proof}

\begin{remark}
The function $w(t,x)$ constructed above is $C^1$ in both variables
but it is $C^2$ in $x$ only for $x\in [0,ct)\cup(ct, k(t)]$; along
$x=ct$, $w_{xx}(t,x)$ has a jumping discontinuity. However, as for
the classical comparison principle, this does not affect the
validity of Lemmas \ref{lem:comp1} and \ref{lem:comp2}.
\end{remark}

\begin{cor}
\label{cor:mono-spreading} If $f$ is of {\rm (f$_M$)} type and
$h_0\geq Z_M=\pi/(2\sqrt{f'(0)})$, then every positive solution
$(u,g,h)$ of \eqref{p} satisfies
\[
(g_\infty,h_\infty)=\R^1 \mbox{ and  $ \lim_{t\to\infty}
u(t,\cdot)=1 $ locally uniformly in $\R^1$.}
\]
\end{cor}
\begin{proof}
 Fix $t_0>0$. Then $g(t_0)<-h_0$, $h(t_0)>h_0$. Set $x_0=[g(t_0)+h(t_0)]/2$ and choose $Z_0>h_0$
such that $[-Z_0 +x_0,Z_0+x_0]\subset (g(t_0), h(t_0))$.  Then
$u(t_0,x +x_0)>0$ in $[-Z_0,Z_0]$.

Since $q^\omega\to 0$ and $z^\omega\to \pi/(2\sqrt{f'(0)})\leq h_0$
as $\omega$ decreases to 0, we can find $\omega>0$ small such that
$z^\omega<Z_0$ and $q^\omega<u(t_0,x+x_0)$ in $[-Z_0,Z_0]$. We now
denote $Z=z^\omega$ and define
\[
v_Z(x):=\left\{\begin{array}{ll} q_0(x+Z;\omega), & x\in [-Z,
0],\\
q_0(-x+Z,\omega), & x\in [0, Z].
\end{array}
\right.
\]
Then it is easily checked that $v_Z$ is a positive solution of
\eqref{v_Z}, and $v_Z(x)\leq q^\omega<u(t_0,x+x_0)$ in $[-Z,Z]$.

Define
\[\tilde u(t,x)=u(t+t_0,x+x_0),\; \tilde g(t)=g(t+t_0)-x_0,\; \tilde
h(t)=h(t+t_0)-x_0.
\]
We find that $\tilde u_0(x):=\tilde u(0,x)>v_Z(x)$ in $[-Z,Z]$ and
$(\tilde u,\tilde g,\tilde h)$ solves \eqref{p} with initial
function $\tilde u_0$. Applying Theorem \ref{thm:spreading} we
deduce that
\[
(\tilde g_\infty,\tilde h_\infty)=\R^1 \mbox{ and  $
\lim_{t\to\infty} \tilde u(t,\cdot)=1 $ locally uniformly in
$\R^1$.}
\]
Clearly this implies the conclusion of the corollary.
\end{proof}

\begin{cor}
\label{cor:mono-vanishing} If $f$ is of  {\rm (f$_M$)} type, then
\[
\lim_{t\to\infty} \|u(t,\cdot)\|_{L^\infty([g(t), h(t)])}=0
\]
implies that $(g_\infty, h_\infty)$ is a finite interval with length
no bigger than $\pi/\sqrt{f'(0)}$.
\end{cor}
\begin{proof}
Otherwise we can find $t_0>0$ such that
\[
h(t_0)-g(t_0)> \pi/\sqrt{f'(0)}.
\]
Let $x_0=[g(t_0)+h(t_0)]/2$ and  define $(\tilde u,\tilde g,\tilde
h)$ by the same formulas  as in the proof of Corollary
\ref{cor:mono-spreading}; we find that the conclusion of Corollary
\ref{cor:mono-spreading} can be applied to $(\tilde u,\tilde
g,\tilde h)$ to deduce that $\tilde u\to 1$ as $t\to\infty$ locally
uniformly in $\R^1$. In view of Lemma \ref{lem:center}, this implies
that
\[
\lim_{t\to\infty} \|u(t,\cdot)\|_{L^\infty([g(t), h(t)])}=1,
\]
a contradiction to our assumption.
\end{proof}

\section{Classification of dynamical behavior and sharp
thresholds}

In this section, based on the results of the previous sections, we
obtain a complete description of the long-time dynamical behavior of
\eqref{p} when $f$ is of monostable, bistable or combustion type. We
also reveal the related but different sharp transition behaviors
between spreading and vanishing for these three types of
nonlinearities.

\subsection{Monostable case}\label{subsec:mono}
Throughout this subsection, we assume that $f$ is of (f$_M$) type.

\begin{thm}\label{thm:mono-dichotomy}{\rm (Dichotomy)}
Suppose that $h_0>0$, $u_0\in \mathscr{X}(h_0)$, and $(u,g,h)$ is
the solution of \eqref{p}.
 Then either spreading happens, namely,
$(g_\infty, h_\infty)=\R^1$ and
\[
\lim_{t\to\infty}u(t,x)=1 \mbox{ locally uniformly in $\R^1$};
\]
or vanishing happens, i.e., $(g_\infty, h_\infty)$ is a finite
interval with length no larger than $\pi/\sqrt{f'(0)}$ and
\[
\lim_{t\to\infty}\max_{g(t)\leq x\leq h(t)} u(t,x)=0.
\]
\end{thm}
\begin{proof}
Since $f$ is of monostable type, it is easy to see that
\eqref{ellip} has no evenly decreasing positive solution, and the
only nonnegative constant solutions are 0 and 1.
 By Theorem \ref{thm:convergence}, we see that in this case the $\omega$ limit
set of $u$ consists of a single constant $0$ or 1. Moreover, if
$(g_\infty, h_\infty)$  a finite interval, then $u(t,x)\to 0$ as
$t\to\infty$ locally uniformly in $(g_\infty, h_\infty)$. In view of
Lemma \ref{lem:center}, this limit implies $\lim_{t\to\infty}
\|u(t,\cdot)\|_{L^\infty([g(t),h(t)])}=0$. Hence we can use
Corollary \ref{cor:mono-vanishing} to conclude that when $(g_\infty,
h_\infty)$ is a finite interval, its length is no larger than
$\pi/\sqrt{f'(0)}$.

It remains to show that when $(g_\infty, h_\infty)=\R^1$, the
$\omega$ limit is 1.  If the limit is $0$, then we can use Corollary
\ref{cor:mono-vanishing} as above to deduce that $(g_\infty,
h_\infty)$ is a finite interval; hence only $\omega(u)=\{1\}$ is
possible.
\end{proof}

\begin{thm}\label{thm:mono-threshold}{\rm (Sharp threshold)}
Suppose that $h_0>0$, $\phi\in \mathscr{X}(h_0)$, and $(u,g,h)$ is a
solution of \eqref{p} with $u_0=\sigma \phi$ for some $\sigma>0$.
Then there exists $\sigma^* = \sigma^* (h_0, \phi) \in [0,\infty]$
such that spreading happens when $  \sigma > \sigma^*$, and
vanishing happens when $0<\sigma\leq \sigma^*$.
\end{thm}

\begin{proof}
By Corollary \ref{cor:mono-spreading}, we find that spreading
happens when $h_0\geq \pi/(2\sqrt{f'(0)})$. Hence in this case we
have $\sigma^*(h_0,\phi)=0$ for any $\phi\in \mathscr{X}(h_0)$.

In what follows we consider the remaining case
$h_0<\pi/(2\sqrt{f'(0)})$. By Theorem \ref{thm:vanishing} (i), we
see that in this case vanishing happens for all small $\sigma>0$.
Therefore
\[
\sigma^*=\sigma^*(h_0,\phi):=\sup \big\{\sigma_0: \mbox{ vanishing
happens for } \sigma\in (0,\sigma_0]\big\}\in(0,+\infty].
\]
 If $\sigma^*=\infty$, then there is nothing left to prove. Suppose
$\sigma^*\in (0,\infty)$. Then by definition
 vanishing happens when $\sigma\in (0,\sigma^*)$, and in view of Theorem
\ref{thm:mono-dichotomy}, there exists a sequence $\sigma_n$
decreasing to $\sigma^*$ such that spreading happens when
$\sigma=\sigma_n$, $n=1,2,\cdots $. For any $\sigma>\sigma^*$, we
can find some $n\geq 1$ such that $\sigma>\sigma_n$. If we denote by
$(u_n, g_n, h_n)$ the solution of \eqref{p} with $u_0=\sigma_n\phi$,
then by the comparison principle, we find that $[g_n(t),
h_n(t)]\subset [g(t), h(t)]$ and $u_n(t,x)\leq u(t,x)$. It follows
that spreading happens for such $\sigma$.

 It remains to show
that vanishing happens when $\sigma=\sigma^*$. Otherwise spreading
must happen when $\sigma=\sigma^*$ and we can find $t_0>0$ such that
$h(t_0)-g(t_0)>\frac{\pi}{\sqrt{f'(0)}}+1$. By the continuous
dependence of the solution of \eqref{p} on its initial values, we
find that if $\epsilon>0$ is sufficiently small, then the solution
of \eqref{p} with $u_0=(\sigma^*-\epsilon)\phi$, denoted by
$(u_\epsilon, g_\epsilon, h_\epsilon)$, satisfies
\[
h_\epsilon(t_0)-g_\epsilon(t_0)>\frac{\pi}{\sqrt{f'(0)}}.
\]
But by Corollary \ref{cor:mono-spreading}, this implies that
spreading happens to $(u_\epsilon, g_\epsilon, h_\epsilon)$, a
contradiction to the definition of $\sigma^*$.
\end{proof}

From the above proof we already know that $\sigma^*(h_0,\phi)=0$ if
$h_0\geq \frac{\pi}{2\sqrt{f'(0)}}$, regardless of the choice of
$\phi\in \mathscr{X}(h_0)$. And if $h_0< \frac{\pi}{2\sqrt{f'(0)}}$,
then $\sigma^*(h_0,\phi)\in(0,+\infty]$. We now investigate when
$\sigma^*(h_0,\phi)$ is finite, and when it is $+\infty$.

For a given $h_0>0$, since any two functions
$\phi_1,\,\phi_2\in\mathscr{X}(h_0)$ can be related by
\[
\sigma_1 \phi_1\leq\phi_2\leq\sigma_2\phi_1
\]
for some positive constants $\sigma_1$ and $ \sigma_2$, we find by
the comparison principle that either $\sigma^*(h_0,\phi)$ is
infinite for all $\phi\in\mathscr{X}(h_0)$, or it is finite for all
such $\phi$. In other words, whether it is finite or not is
determined by $h_0$ and $f$, but not affected by the choice of
$\phi\in\mathscr{X}(h_0)$.

\begin{prop}\label{prop:mono-sigma-finite}
Assume that $h_0 < \pi/(2 \sqrt{f'(0)})$ and  $f(u)\geq - Lu$ for
all $u>0$ and some $L>0$.  Then $\sigma^*(h_0,\phi)\in (0,\infty)$
for all $\phi\in\mathscr{X}(h_0)$.
\end{prop}

\begin{proof} Fix an arbitrary $\phi\in \mathscr{X}(h_0)$. By
Theorem \ref{thm:mono-threshold}, it suffices to show that spreading
happens when $u_0=\sigma\phi$ and $\sigma$ is large. We will achieve
this by constructing a suitable lower solution.

We start with the following Sturm-Liouville eigenvalue problem
\begin{equation}\label{SL-half}
\left\{
 \begin{array}{l}
 \varphi''(x) +\frac{1}{2}\varphi'(x) +\lambda \varphi (x) =0,\quad x\in (0,1),\\
 \varphi' (0) =\varphi (1)=0.
 \end{array}
\right.
\end{equation}
It is well known that the first eigenvalue $\lambda_1$ of this
problem is simple and the corresponding first eigenfunction
$\varphi_1 (x)$ can be chosen positive in $[0,1)$. Moreover, one can
easily show that $\lambda_1>\frac{1}{16}$ and  $\varphi'_1 (x) <0$
for $x\in (0,1]$. We assume further that $\|\varphi_1\|_{L^\infty
([0,1])} =\varphi_1 (0)=1$.

We extend $\varphi_1$ to $[-1,1]$ as an even function. Then clearly
\begin{equation}\label{SL}
\left\{
 \begin{array}{l}
 \varphi_1''(x) +\frac{{\rm sgn}(x)}{2}\varphi_1'(x) +\lambda_1
 \varphi_1 (x) =0,\quad x\in (-1,1),\\
 \varphi_1 (-1) =\varphi_1 (1)=0.
 \end{array}
\right.
\end{equation}

We now choose constants $\varepsilon, \bar{Z}, T, \lambda, \rho$ in
the following way:
 \[
 0<\varepsilon < \min \{1, h^2_0\},\;
 \bar{Z}:= 1+ \pi/(2\sqrt{f'(0)}),\; T> {\bar{Z}}^2,
\]
  and
\begin{equation}\label{def:alpha}
\lambda > \lambda_1 + L (T+1),
\end{equation}
\begin{equation}\label{def:rho-1}
- 2\mu \rho \varphi'_1 (1) > (T+1)^\lambda .
\end{equation}

Define
\begin{equation}\label{lower-sol}
w(t,x) := \frac{\rho}{(t+\varepsilon)^\lambda} \ \varphi_1 \left(
\frac{x}{\sqrt{t+\varepsilon}}\right) \quad \mbox{for } x\in
[-\sqrt{t+\varepsilon},\sqrt{t+\varepsilon}],\ t\geq 0.
\end{equation}
We show that $(w(t,x), - \sqrt{t+\varepsilon},
\sqrt{t+\varepsilon})$ is a lower solution of \eqref{p} on the
time-interval $[0,T]$.

In fact, for $x\in (-\sqrt{t+\varepsilon},\sqrt{t+\varepsilon})$ and
$t\in [0,T]$ we have
\begin{eqnarray*}
w_t -w_{xx} - f(w) & \leq & w_t -w_{xx} +L w \\
& = & \frac{-\rho}{(t+\varepsilon)^{\lambda+1}} \left[ \varphi''_1
+\frac{x}{2\sqrt{t+\varepsilon}} \varphi'_1  +(\lambda -
L(t+\varepsilon)) \varphi_1 \right]\\
& \leq & \frac{-\rho}{(t+\varepsilon)^{\lambda+1}} \left[
\varphi''_1 +\frac{{\rm sgn}(x) }{2} \varphi'_1  +(\lambda-
L(t+\varepsilon))
\varphi_1 \right] \\
& \leq & \frac{-\rho}{(t+\varepsilon)^{\lambda+1}} \left[
\varphi''_1 +\frac{{\rm sgn}(x)}{2} \varphi'_1  + \lambda_1
\varphi_1 \right] =0.
\end{eqnarray*}
Clearly $w(t,\pm\sqrt{t+\varepsilon})=0$, and by \eqref{def:rho-1}
we have
$$
(\sqrt{t+\varepsilon})' \pm \mu w_x (t, \pm\sqrt{t+\varepsilon})
\leq \frac{1}{2\sqrt{t+\varepsilon}} \left( 1 + \frac{2\mu
\rho}{(T+1)^\lambda } \varphi'_1 (1) \right) <0.
$$
Finally, since $\varepsilon <h^2_0$ we can choose $\hat{\sigma}
>0$ large such that
$$
w(0,x) = \frac{\rho}{\varepsilon^\lambda} \varphi_1 \Big(
\frac{x}{\sqrt{\varepsilon}}\Big) < \hat{\sigma} \phi(x)\quad
\mbox{for } x\in [-\sqrt{\varepsilon}, \sqrt{\varepsilon}]\subset
[-h_0, h_0].
$$

Hence $(w(t,x), -\sqrt{t+\varepsilon}, \sqrt{t+\varepsilon})$ is a
lower solution of \eqref{p} over the time interval $[0,T]$ if in
\eqref{p} we take $u_0(x)={\sigma} \phi(x)$ with $\sigma\geq
\hat\sigma$.  It follows that if $(u,g,h)$ is the solution of
\eqref{p} with $u_0=\sigma \phi$ and $\sigma\geq \hat\sigma$, then
$$
g(t) \leq -\sqrt{t+\varepsilon},\; h(t) \geq \sqrt{t+\varepsilon}
\quad \mbox{for } t\in [0,T].
$$
In particular, $h(T)-g(T) > 2\sqrt{T} > 2\bar{Z} >
\pi/\sqrt{f'(0)}$. So spreading happens by Corollary
\ref{cor:mono-spreading} for such $(u,g,h)$.
\end{proof}

\begin{prop}\label{prop:mono-sigma-infty}
Assume that
\begin{equation}\label{growth M1}
\lim\limits_{s\rightarrow \infty} \frac{-f(s)}{s^{1+2\beta}} =
\infty,
\end{equation}
for some
\begin{equation}\label{cond alpha beta}
\beta > \frac{3+\sqrt{13}}{2}.
\end{equation}
 Then there exists $ Z^0_M \in (0, \pi/(2\sqrt{f'(0)})) $ such
that for every $\phi\in\mathscr{X}(h_0)$,
\begin{itemize}
\item [\rm (i)]  $\sigma^* (h_0,\phi) =\infty$ if $h_0 \in (0,
Z^0_M]$, and
\item [\rm (ii)]
$\sigma^* (h_0, \phi) \in (0, \infty)$ if $h_0 \in \big( Z^0_M ,
\pi/(2\sqrt{f'(0)}) \big)$.
\end{itemize}
\end{prop}

The proof of this result is rather technical and is postponed to the
end of this section. Clearly Theorem \ref{thm:mono} follows from
Theorems \ref{thm:mono-dichotomy}, \ref{thm:mono-threshold} and
Proposition \ref{prop:mono-sigma-finite} above.

\subsection{Bistable case}\label{subsec:bistable}
Throughout this subsection, we assume that $f$ is of (f$_B$) type.

\begin{thm}\label{thm:bi-trichotomy}
{\rm (Trichotomy)}  Suppose that $h_0>0$, $u_0 \in \mathscr
{X}(h_0)$ and $(u,g,h)$ is the solution of \eqref{p}. Then either

{\rm (i) Spreading:} $(g_\infty, h_\infty)=\R^1$ and
\[
\lim_{t\to\infty}u(t,x)=1 \mbox{ locally uniformly in $\R^1$},
\]
or

{\rm (ii) Vanishing:} $(g_\infty, h_\infty)$ is a finite interval
and
\[
\lim_{t\to\infty}\max_{g(t)\leq x\leq h(t)} u(t,x)=0,
\]
or

{\rm (iii) Transition:} $(g_\infty, h_\infty)=\R^1$ and there exists
a continuous function $\gamma: [0,\infty)\to [-h_0,h_0]$ such that
\[
\lim_{t\to\infty}|u(t,x)-v_\infty(x+\gamma(t))|=0 \mbox{ locally
uniformly in $\R^1$},
\]
where $v_\infty$ is the unique positive solution to
\[v''+f(v)=0 \; (x\in\R^1),\; v'(0)=0,\; v(-\infty)=v(+\infty)=0.\]
\end{thm}
\begin{proof} By Theorem \ref{thm:convergence}, we have either
$(g_\infty, h_\infty)$ is a finite interval or $(g_\infty,
h_\infty)=\R^1$. In the former case,  $\lim_{t\to\infty}u(t,x)=0$
locally uniformly in $(g_\infty, h_\infty)$, which, together with
Lemma \ref{lem:center}, implies that (ii) holds.

Suppose now  $(g_\infty, h_\infty)=\R^1$;  then either
$\lim_{t\to\infty} u(t,x)$ is a nonnegative constant solution of
\begin{equation}
\label{ellip-R10} v_{xx}+f(v)=0 \mbox{ in } \R^1,
\end{equation}
or
\[
u(t,x)-v(x+\gamma(t))\to 0 \mbox{ as } t\to\infty \mbox{ locally
uniformly in } \R^1,
\]
where $v$ is an evenly decreasing positive solution of
\eqref{ellip-R10}, and $\gamma:[0,\infty)\to [-h_0,h_0]$ is a
continuous function.

Since $f$ is of bistable type, it is well-known (see \cite{DM}) that
bounded nonnegative solutions of \eqref{ellip-R10} consist of the
following:
\begin{itemize}
\item[\vspace{10pt}(1)] \;constant solutions: 0,\ $\theta$, 1;
\item[\qquad(2)] \;a family of periodic solutions satisfying
$0<\min v <\theta<\max v <\overline \theta ; $
\item[\qquad(3)] \; a family of symmetrically decreasing solutions
$v_\infty(\cdot-a),\;a\in \R^1$, where $v_\infty$ is uniquely
determined by
\[
v_\infty''+f(v_\infty)=0 \ \hbox{ in } \R^1, \quad
v_\infty(0)=\overline\theta,\quad v_\infty'(0)=0,
\]
which necessarily satisfies $\lim_{|x|\to\infty}v_\infty(x)=0$.
\end{itemize}
From this list, clearly only $0,\; \theta,\; 1$ and
$v_\infty(\cdot-a)$ are possible members of $\omega(u)$.

By Corollary \ref{cor:BC-vanishing-finite}, $\omega(u)=\{0\}$ is
impossible. It remains to show that $\omega(u)\not=\{\theta\}$. We
argue indirectly by assuming that $u(t,x)\to\theta$ as $t\to\infty$
locally uniformly in $\R^1$. Let $v_0(x)$ be a periodic solution of
\eqref{ellip-R10} as given in (2) above. We now consider the number
of zeros of the function
\[
w(t,x):=u(t,x)-v_0(x)
\]
in the interval $[g(t), h(t)]$, and denote this number by
$\mathcal{Z}(t)$. Clearly $w(t, g(t))<0$ and $w(t, h(t))<0$ for all
$t>0$. Therefore we can use the zero number result of \cite{A} to
the equation satisfied by $w$ to conclude that $\mathcal{Z}(t)$ is
finite and non-increasing in $t$ for $t>0$. (We could use a change
of variable to change the varying interval $[g(t), h(t)]$ into a
fixed one, and then use \cite{A} to the reduced equation.) On the
other hand, since $u(t,x)\to\theta$ and $v_0(x)$ oscillates around
$\theta$, we find that $\mathcal{Z}(t)\to\infty$ as $t\to\infty$.
This contradiction shows that $v=\theta$ is impossible. The proof is
complete.
\end{proof}

\begin{thm}\label{thm:bi-threshold}{\rm (Sharp threshold)}
Suppose that $h_0>0$, $\phi\in \mathscr{X}(h_0)$, and $(u,g,h)$ is a
solution of \eqref{p} with $u_0=\sigma \phi$ for some $\sigma>0$.
Then there exists $\sigma^* = \sigma^* (h_0, \phi) \in (0,\infty]$
such that spreading happens when $  \sigma > \sigma^*$,  vanishing
happens when $0<\sigma< \sigma^*$, and transition happens when
$\sigma=\sigma^*$.
\end{thm}

\begin{proof}
By Theorem \ref{thm:vanishing} (ii) we find that vanishing happens
 if $\sigma < \theta/ \|\phi\|$. Hence
\[
\sigma^*=\sigma^*(h_0,\phi):=\sup \big\{\sigma_0: \mbox{ vanishing
happens for } \sigma\in (0,\sigma_0]\big\}\in(0,+\infty].
\]
If $\sigma^*=+\infty$, then there is nothing left to prove. So we
assume that $\sigma^*$ is a finite positive number.

By definition, vanishing happens for all $\sigma\in (0,\sigma^*)$.
We now consider the case $\sigma=\sigma^*$. In this case, we cannot
have vanishing, for otherwise we have, for some large $t_0>0$,
$u(t_0,x)<\theta$ in $[g(t_0), h(t_0)]$, and due to the continuous
dependence of the solution on the initial values,  we can find
$\epsilon>0$ sufficiently small such that the solution $(u_\epsilon,
g_\epsilon, h_\epsilon)$ of \eqref{p} with
$u_0=(\sigma^*+\epsilon)\phi$ satisfies
\[
u_\epsilon(t_0, x)<\theta \mbox{ in } [g_\epsilon(t_0),
h_\epsilon(t_0)].
\]
 Hence we can apply Theorem \ref{thm:vanishing} (ii) to
conclude that vanishing happens to $(u_\epsilon, g_\epsilon,
h_\epsilon)$, a contradiction to the definition of $\sigma^*$. Thus
at $\sigma=\sigma^*$ either spreading or transition happens.

We show next that spreading cannot happen at $\sigma=\sigma^*$.
Suppose this happens. Let $v_Z$ be a stationary solution as given in
Lemma \ref{lem:v_Z}. Then we can find $t_0>0$ large such that
\begin{equation}
\label{u-v_Z}
 [-Z,Z]\subset (g(t_0), h(t_0)),\; u(t_0, x)>v_Z(x) \mbox{ in
} [-Z,Z].
\end{equation}
By the continuous dependence of the solution on initial values, we
can find a small $\epsilon>0$ such that the solution $(u^\epsilon,
g^\epsilon, h^\epsilon)$ of \eqref{p} with
$u_0=(\sigma^*-\epsilon)\phi$ satisfies \eqref{u-v_Z}, and by
Theorem \ref{thm:spreading}, spreading happens for $(u^\epsilon,
g^\epsilon, h^\epsilon)$. But this is a contradiction to the
definition of $\sigma^*$.

Hence transition must happen when $\sigma=\sigma^*$. We show next
that spreading happens when $\sigma>\sigma^*$. Let $(u,g,h)$ be a
solution of \eqref{p} with some $\sigma>\sigma^*$, and denote the
solution of \eqref{p} with $\sigma=\sigma^*$ by $(u_*,g_*,h_*)$. By
the comparison theorem we know that
\[
[g_*(1), h_*(1)]\subset (g(1), h(1)),\; u_*(1,x)<u(1,x) \mbox{ in }
[g_*(1),h_*(1)].
\]
Hence we can find $\epsilon_0>0$ small such that for all
$\epsilon\in [0,\epsilon_0]$,
\[
[g_*(1)-\epsilon, h_*(1)-\epsilon]\subset (g(1), h(1)),\;
u_*(1,x+\epsilon)<u(1,x) \mbox{ in }
[g_*(1)-\epsilon,h_*(1)-\epsilon].
\]
Now define
\[
u_\epsilon(t,x)=u_*(t+1, x+\epsilon),\;
g_\epsilon(t)=g_*(t+1)-\epsilon,\; h_\epsilon(t)=h_*(t+1)-\epsilon.
\]
Clearly $(u_\epsilon, g_\epsilon, h_\epsilon)$ is a solution of
\eqref{p} with $u_0(x)=u_*(1, x+\epsilon)$. By the comparison
principle we have, for all $t>0$ and $\epsilon\in (0,\epsilon_0]$,
\[
[g_\epsilon(t), h_\epsilon(t)]\subset (g(t+1), h(t+1)), \;
u_\epsilon(t,x)\leq u(t+1,x) \mbox{ in } [g_\epsilon(t),
h_\epsilon(t)].
\]

If $u_*(t,x)-v_\infty(x+\gamma(t))\to 0$ as $t\to\infty$ and
$\omega(u)\not=\{1\}$, then necessarily $u(t,x)-v_\infty(x+\tilde
\gamma(t))\to 0$ as $t\to\infty$. Here both limits are locally
uniform in $\R^1$, and $\gamma,\tilde\gamma$ are continuous
functions from $[0,\infty)$ to $[-h_0,h_0]$.

On the other hand, the above inequalities imply that
\[
\limsup
_{t\to\infty}[v_\infty(x+\epsilon+\gamma(t))-v_\infty(x+\tilde
\gamma(t))]\leq 0
\]
 for all $x\in\R^1$ and $\epsilon\in
(0,\epsilon_0]$. Since $v_\infty$ is an evenly decreasing function,
this implies that $\lim_{t\to\infty}[\epsilon+\gamma(t)-\tilde
\gamma(t)]=0$ for every $\epsilon\in (0,\epsilon_0]$. Clearly this
is impossible. Thus we must have $\omega(u)=\{1\}$. This proves that
spreading happens for $\sigma>\sigma^*$.
\end{proof}

Next we determine when $\sigma^*(h_0,\phi)$ is finite and when it is
infinite.

\begin{prop}
\label{prop:bi-sigma-finite} Let $Z_B$ be given by \eqref{def:RB1}.
Then $\sigma^*(h_0,\phi)<\infty$ for all $\phi\in \mathscr{X}(h_0)$
if $h_0\geq Z_B$, or if $h_0\in (0, Z_B)$ and $f(u)\geq -Lu$ for all
$u>0$ and some $L>0$.
\end{prop}
\begin{proof}
Let $h_0 >0$, $\phi \in \mathscr {X}(h_0)$ and $(u,g,h)$ be a
solution of \eqref{p} with $u_0=\sigma\phi$. It suffices to show
that spreading happens for all large $\sigma$ under the given
conditions.

 First we suppose that $h_0\geq Z_B$. By Lemma \ref{lem:v_Z},
 \eqref{v_Z} has a positive solution $v_Z$ with $Z=h_0$.
 For sufficiently
large $\sigma >0$ clearly $\sigma\phi\geq v_Z$. Thus we can apply
Theorem \ref{thm:spreading} to conclude that spreading happens for
$(u,g,h)$ with such $\sigma$, as we wanted.

Next we consider the case that $h_0\in (0, Z_B)$ and $f(u)\geq -Lu$
for all $u>0$ and some $L>0$. In this case we construct a lower
solution as in the proof of Proposition \ref{prop:mono-sigma-finite}
with the following changes: $\bar Z\geq 1+\pi/(2\sqrt{f'(0)})$ is
replaced by $\bar Z\geq 1+ Z_B$, and we add a further restriction
for $\rho$, namely
\[
\frac{\rho}{(T+\varepsilon)^\lambda}\varphi_1\Big(\frac{x}{\sqrt{T+\varepsilon}}\Big)\geq
v_{Z_B}(x) \mbox{ in } [-Z_B,Z_B].
\]
We deduce as in the proof of Proposition
\ref{prop:mono-sigma-finite} that, for $\sigma\geq \hat\sigma$,
\[
h(T)-g(T)>2Z_B,\; u(T,x)\geq w(T,x)\geq v_{Z_B} \mbox{ in } [-Z_B,
Z_B].
\]
Then by  Theorem \ref{thm:spreading}, we deduce that spreading
happens for $(u,g,h)$ with $\sigma\geq \hat\sigma$. The proof is
complete.
\end{proof}

Clearly Theorem \ref{thm:bi} is a consequence of Theorems
\ref{thm:bi-trichotomy}, \ref{thm:bi-threshold} and Proposition
\ref{prop:bi-sigma-finite}. The following result gives conditions
for $\sigma^*(h_0,\phi)=\infty$, whose proof will be given in the
last subsection of this section.

\begin{prop}\label{prop:bi-sigma-infty}
  Assume that
\begin{equation}\label{growth rate}
\lim\limits_{s\rightarrow \infty} \frac{-f(s)}{s^{1+2\beta}} =
\infty \quad \mbox{for some } \beta >2.
\end{equation}
Then there exists $ Z^0_B \in (0, Z_B) $ such that, for every
$\phi\in \mathscr{X}(h_0)$,
\begin{itemize}
\item [\rm (i)]  $\sigma^* (h_0,\phi) =\infty$ if $h_0 \leq Z^0_B$, and
\item [\rm (ii)] $\sigma^* (h_0, \phi) \in (0, \infty)$ if $h_0 > Z^0_B$.
\end{itemize}
\end{prop}

\subsection{Combustion case}\label{subsec:combustion}

Throughout this subsection, we assume that $f$ is of (f$_C$) type.

\begin{thm}\label{thm:combus-trichotomy}
{\rm (Trichotomy)}  Suppose that $h_0>0$, $u_0 \in \mathscr
{X}(h_0)$ and $(u,g,h)$ is the solution of \eqref{p}. Then either

{\rm (i) Spreading:} $(g_\infty, h_\infty)=\R^1$ and
\[
\lim_{t\to\infty}u(t,x)=1 \mbox{ locally uniformly in $\R^1$},
\]
or

{\rm (ii) Vanishing:} $(g_\infty, h_\infty)$ is a finite interval
and
\[
\lim_{t\to\infty}\max_{g(t)\leq x\leq h(t)} u(t,x)=0,
\]
or

{\rm (iii) Transition:} $(g_\infty, h_\infty)=\R^1$ and
\[
\lim_{t\to\infty}u(t,x)=\theta.
\]
\end{thm}
\begin{proof}
One easily sees that bounded nonnegative solutions of
\begin{equation}
\label{ellip-R1} v_{xx}+f(v)=0 \mbox{ in } \R^1,
\end{equation}
 with a combustion type $f$, consists of the following constant
solutions only:
 0,\ every $c\in(0,\theta)$, \ $\theta$,\ 1.

Therefore, by Theorem \ref{thm:convergence}, we have either
$(g_\infty, h_\infty)$ is a finite interval and
$\lim_{t\to\infty}u(t,x)=0$ locally uniformly in $(g_\infty,
h_\infty)$, or $(g_\infty, h_\infty)=\R^1$ and $\lim_{t\to\infty}
u(t,x)=v$ locally uniformly in $\R^1$, with $v$ a constant
nonnegative solution of \eqref{ellip-R1}. As before we can use Lemma
\ref{lem:center} to conclude that when $(g_\infty, h_\infty)$ is a
finite interval, then vanishing happens.

It remains to show that when $(g_\infty, h_\infty)=\R^1$, then
$v\equiv 1$ or $v\equiv \theta$.   As before Corollary
\ref{cor:BC-vanishing-finite}
 shows that $v=0$ is impossible when $(g_\infty, h_\infty)=\R^1$.
 We show next that $v\not =c$ for any $c\in (0,\theta)$.
Suppose by way of contradiction that $v\equiv c\in (0,\theta)$. Then
in view of Lemma \ref{lem:center}, for some large $t_0>0$ we have
$\|u(t_0,\cdot)\|_{L^\infty}<\theta$. Thus we can apply Theorem
\ref{thm:vanishing} to conclude that vanishing happens to $(u,g,h)$,
a contradiction to $\omega(u)=\{c\}$.

 The proof is complete.
\end{proof}

\begin{thm}\label{thm:combus-threshold}{\rm (Sharp threshold)}
Suppose that $h_0>0$, $\phi\in \mathscr{X}(h_0)$, and $(u,g,h)$ is a
solution of \eqref{p} with $u_0=\sigma \phi$ for some $\sigma>0$.
Then there exists $\sigma^* = \sigma^* (h_0, \phi) \in (0,\infty]$
such that spreading happens when $  \sigma > \sigma^*$,  vanishing
happens when $0<\sigma< \sigma^*$, and transition happens when
$\sigma=\sigma^*$.
\end{thm}

\begin{proof} The proof is identical to that of Theorem
\ref{thm:bi-threshold} except the last part, where it shows that
spreading happens when $\sigma>\sigma^*$. This part has to be proved
differently.

Let $(u_*,g_*,h_*)$ be the solution of \eqref{p} with $u_0=\sigma^*
\phi$, and $(u,g,h)$ a solution with $u_0=\sigma \phi$ and
$\sigma>\sigma^*$. Since $u_*(t,x)\to \theta$ locally uniformly in
$\R^1$ as $t\to\infty$, in view of Lemma \ref{lem:center}, we can
find $T>0$ large such that
\begin{equation}
\label{u*-T} u_*(t,x)<\theta+\delta_0/2 \mbox{ for } t\geq T/2,\;
x\in [g_*(t), h_*(t)],
\end{equation}
where $\delta_0$ is given in \eqref{combus1}.
 By the comparison principle we have
\begin{equation}
\label{u*-u-T}
 [g_*(T), h_*(T)]\subset (g(T), h(T)),\;
u_*(T,x)<u(T,x) \mbox{ in } [g_*(T),h_*(T)].
\end{equation}
Now, for $\xi\in(0,1)$, we define
\[
v^\xi(t,x):=\xi^{-1}\:\! u_*(\xi\:\!t,\sqrt{\xi}\:\! x),\;
g^\xi(t)=\xi^{-1/2}g_*(\xi t),\; h^\xi(t)=\xi^{-1/2}h_*(\xi t).
\]
Then, by \eqref{u*-T} and \eqref{u*-u-T}, we can choose
$\xi_0\in(0,1)$ close enough to $1$ so that, for every
$\xi\in[\xi_0,1)$,
\begin{equation}\label{vsigma-bound}
 v^\xi(t,x)\leq \theta+\delta_0 \quad\ \
\hbox{for all}\ \ t\geq T,\; x\in [g^\xi(t), h^\xi(t)],
\end{equation}
and
\begin{equation}
\label{vsigma-u-T} [g^{\xi}(T), h^\xi(T)]\subset (g(T), h(T)),\;
v^\xi(T,x)\leq u(T,x) \mbox{ in } [g^\xi(T), h^\xi(T)].
\end{equation}
 Observe that $v^\xi$ satisfies the equation
\[
v^\xi_t=v^\xi_{xx}+f(\xi v^\xi) \mbox{ for } t>T, \; x\in [g^\xi(t),
h^\xi(t)].
\]
By \eqref{vsigma-bound} and \eqref{combus1}, we have $f(\xi
v^\xi)\leq f(v^\xi)$. Therefore in view of \eqref{vsigma-u-T}, we
find that $(v^\xi, g^{\xi}, h^{\xi})$ is a lower solution of
\eqref{p} for $t\geq T$. It follows that
\[
u(t,x)\geq v^\xi(t,x) \mbox{ and } v\geq \lim_{t\to\infty}
v^\xi(t,x)=\theta/\xi,
\]
where $v$ is the $\omega$-limit of $u$. Thus we must have $v\equiv
1$, as we wanted.
\end{proof}

\begin{prop}
\label{prop:combus-sigma-finite} Let $Z_C$ be given by
\eqref{def:RC1}. Then $\sigma^*(h_0,\phi)<\infty$ for all $\phi\in
\mathscr{X}(h_0)$ if $h_0\geq Z_C$, or if $h_0\in (0, Z_C)$ and
$f(u)\geq -Lu$ for all $u>0$ and some $L>0$.
\end{prop}

\begin{prop}\label{prop:combus-sigma-infty}
  Assume that $f$ satisfies
\eqref{growth rate}. Then there exists $ Z^0_C \in (0, Z_C) $ such
that, for every $\phi\in \mathscr{X}(h_0)$,
\begin{itemize}
\item [\rm (i)]  $\sigma^* (h_0, \phi) =\infty$ if $h_0 \leq Z^0_C$, and
\item [\rm (ii)] $\sigma^* (h_0, \phi) \in (0, \infty)$ if $h_0 > Z^0_C$.
\end{itemize}
\end{prop}

The proof of Proposition \ref{prop:combus-sigma-finite} is identical
to that of Proposition \ref{prop:bi-sigma-finite}; all we need is to
replace $Z_B$ by $Z_C$ in the proof. The proof of Proposition
\ref{prop:combus-sigma-infty} is given in the next subsection.
Evidently, Theorem \ref{thm:combus} is a consequence of Theorems
\ref{thm:combus-trichotomy}, \ref{thm:combus-threshold} and
Proposition \ref{prop:combus-sigma-finite}.

\subsection{Proof of Propositions \ref{prop:mono-sigma-infty}, \ref{prop:bi-sigma-infty}
and \ref{prop:combus-sigma-infty}}\label{subsec:sigma=infty}

In this subsection we always assume that $f$ is of (f$_M$), or
(f$_B$), or (f$_C$) type. We will prove Propositions
\ref{prop:bi-sigma-infty} and \ref{prop:combus-sigma-infty} first,
and then prove Proposition \ref{prop:mono-sigma-infty}.

If $f$  satisfies
\begin{equation}\label{growth rate 1}
\lim\limits_{u\rightarrow \infty} \frac{-f(u)}{u^{1+2\beta}} =
\infty \quad \mbox{for some }  \beta >2,
\end{equation}
then taking
\begin{equation}\label{def-L}
L=L(\beta) = \frac{1+\beta}{\beta^2}\cdot 2^{1+2\beta},
\end{equation}
we can find $s = s(\beta)
>1$ such that
\begin{equation}\label{growth rate 2}
 -f(u) \geq L u^{1+2\beta} \quad \mbox{for } u\geq s.
\end{equation}

Let $h_0>0$, and $(u,g,h)$ be the solution of \eqref{p} with
$u_0=\phi\in \mathscr{X}(h_0)$.  We show that $\int_{g(t)}^{h(t)}
u(t,x; \phi)$ can be made as small as we want if  $h_0$ is small
enough and $t$ is chosen suitably, regardless of the choice of
$\phi$.

\begin{lem}\label{lem:int-small}
Suppose that \eqref{growth rate 1} or \eqref{growth rate 2} holds.
Then given any $\varepsilon >0$ we can find
$h_0^*=h^*_0(\varepsilon)>0$ such that, for each $h_0 \in (0,h^*_0]$
there exists $t_0=t_0(h_0)>0$ so that
$$
\int_{g(t_0)}^{h(t_0)} u(t_0, x; \phi) dx < \varepsilon \mbox{ for
all $\phi \in \mathscr {X}(h_0)$.}
$$
\end{lem}

\begin{proof} For any $h_0 \in (0,1)$, set
$$
\psi(x) = (x-h_0)^{-\frac{1}{\beta}} - h_0^{-\frac{1}{\beta}}\quad
\mbox{for } h_0 < x \leq 2h_0.
$$
For $h_0 + ct <x\leq 2h_0 +ct,\ t>0$, we define
\begin{equation}\label{upper sol for infinity}
w (t,x):= \psi(x-ct),\quad k(t)= 2h_0 +ct \quad \mbox{with }
c=\frac{\mu}{\beta h_0^{\frac{1+\beta}{\beta}}}.
\end{equation}

With $s$ given by \eqref{growth rate 2}, we set
\begin{equation}\label{def-epsilon1}
\varepsilon_1 := h_0  (1+s)^{-\beta}\ (<h_0) \quad (\mbox{or
equivalently, } \psi (h_0 +\varepsilon_1) = s
h_0^{-\frac{1}{\beta}}).
\end{equation}

\smallskip

We now consider $w$ for $h_0+ct+\varepsilon_1 \leq x \leq 2h_0 +ct,\
t>0$. It is easily seen that
$$
w_t (t,x) = -c\psi'(x-ct) = \frac{c}{\beta (x-h_0
-ct)^{\frac{1+\beta}{\beta}}} \geq \frac{c}{\beta
h_0^{\frac{1+\beta}{\beta}}},
$$
$$
w_{xx}  (t,x)= \frac{1+\beta}{\beta^2
(x-h_0-ct)^{\frac{1+2\beta}{\beta}}} \leq \frac{1+\beta}{\beta^2}
\cdot \frac{(1+s)^{1+2\beta}}{h_0^{\frac{1+2\beta}{\beta}}}.
$$
Hence, with $F:= \sup_{0\leq \xi <\infty} f(\xi)$,
\[
\begin{aligned}
w_t -w_{xx} -f(w) \geq &\frac{1}{\beta h_0^{\frac{1
+2\beta}{\beta}}} \left[ ch_0 -\frac{1+\beta}{\beta}
(1+s)^{1+2\beta} - F\beta h_0^{\frac{1+2\beta}{\beta}}\right]\\
=&\frac{1}{\beta h_0^{\frac{1 +2\beta}{\beta}}} \left[
\frac{\mu}{\beta}h_0^{-1/\beta} -\frac{1+\beta}{\beta}
(1+s)^{1+2\beta} - F\beta h_0^{\frac{1+2\beta}{\beta}}\right]
 \geq 0,
 \end{aligned}
\]
provided $h_0$ is sufficiently small.

\smallskip

Next we consider $w$ for $h_0+ ct <x \leq h_0 + ct+\varepsilon_1,\
t>0$. In this range, we have
\begin{eqnarray*}
w (t,x) &  \geq& \frac{1}{(x-h_0 -ct)^{\frac{1}{\beta}}} \left( 1 -
\Big(\frac{\varepsilon_1}{h_0}
\Big)^{\frac{1}{\beta}}\right) \\
& = &  \frac{1}{(x-h_0 -ct)^{\frac{1}{\beta}}} \Big(\frac{s}{1+s}
\Big)\\
& \geq &\frac{s}{\varepsilon_1^{1/\beta}(1+s)}= \frac{s}{
h_0^{\frac{1}{\beta}}}
>s.
\end{eqnarray*}
Thus,
\begin{eqnarray*}
w_t -w_{xx} -f(w) & \geq & -c\psi'(x-ct) -\psi'' (x-ct) +L
[\psi (x-ct)]^{1+2\beta}\\
& \geq & L\left[ \frac{1}{(x-h_0 -ct)^{\frac{1}{\beta}}} \Big(
\frac{s}{1+s}\Big) \right]^{1+2\beta} -\frac{1+\beta}{\beta^2
(x-h_0 -ct)^{\frac{1 +2\beta}{\beta}}}\\
& = & \frac{1}{(x-h_0 -ct)^{\frac{1 +2\beta}{\beta}}} \left[
L\Big(\frac{s}{1+s}\Big)^{1+2\beta} -\frac{1+\beta}{\beta^2}
\right] \\
& > & \frac{1}{(x-h_0 -ct)^{\frac{1 +2\beta}{\beta}}} \left[
L\Big(\frac{1}{2}\Big)^{1+2\beta} -\frac{1+\beta}{\beta^2} \right] =
0.
\end{eqnarray*}
Clearly,
$$
k'(t) +\mu w_x (t, k(t)) = c-\frac{\mu}{\beta
h_0^{\frac{1+\beta}{\beta}}} =0.
$$

\smallskip
We now compare $(u, h)$ with $(w, k)$ over the region
\[
\Omega:=\{(t,x): h_0+ct\leq x\leq k(t)\}\cap \{(t,x): 0\leq x\leq
h(t)\}.
\]
By definition,
\[ u(t,x)=0 \mbox{ for } x=h(t),\; w(t,x)=+\infty \mbox{ for }
x=h_0+ct.
\]
Thus we can apply Lemma \ref{lem:comp2} to deduce that whenever
$J(t):=\{x: (t,x)\in\Omega\}$ is nonempty, we have $h(t)\leq k(t)$
and $u(t,x)\leq w(t,x)$ in $J(t)$. Thus we have $h(t)\leq k(t)$ for
all $t>0$.

By \eqref{growth rate 2} and the definition of $c$,
$$
\tau_1 := \int_{\frac{s}{h_0^{\frac{1}{\beta}}}}^{\infty}
\frac{dr}{-f(r)} \leq \frac{h_0^2}{2\beta L s^{2\beta}},\;\;\;
c\tau_1 \leq \frac{\mu}{2\beta^2 L s^{2\beta}} h_0^{\frac{\beta
-1}{\beta}}.
$$

Let $\zeta(t)$ be the solution of
$$
\zeta' (t) = f(\zeta),\quad \zeta(0) = \|\phi\|_{L^\infty}+1.
$$
Then $u(t,x;  \phi) \leq \zeta(t)$ for $t\geq 0$. We claim that
$u(t,x; \phi) \leq sh_0^{-\frac{1}{\beta}}$ for $t\geq \tau_1$.

Indeed, since $f(1)=0$ and $f(\xi)<0$ for $\xi>1$, we find that
$\zeta(t)>1$ and is decreasing for $t>0$. Moreover,
\begin{eqnarray*}
\tau_1&=&\int_{\|\phi\|_\infty+1}^{\zeta(\tau_1)}\frac{d\zeta}{f(\zeta)}\\
&=&\int_{\|\phi\|_\infty+1}^\infty\frac{d\zeta}{f(\zeta)}-\int_{\zeta(\tau_1)}^\infty
\frac{d\zeta}{f(\zeta)}\\
&<&\int_{\zeta(\tau_1)}^\infty \frac{d\zeta}{-f(\zeta)}.
\end{eqnarray*}
Thus
\[
\int_{s/h_0^{1/\beta}}^\infty
\frac{d\zeta}{-f(\zeta)}<\int_{\zeta(\tau_1)}^\infty
\frac{d\zeta}{-f(\zeta)}.
\]
It follows that $\zeta(\tau_1)<s/h_0^{1/\beta}$ and hence
$\zeta(t)<s/h_0^{1/\beta}$ for all $t\geq \tau_1$, which implies the
claim.

By this estimate of $u(t,x;\phi)$ we obtain, for $t=\tau_1$,
$$
\int_0^{h(\tau_1)} u(\tau_1 , x; \phi) dx  \leq  (2h_0 +c\tau_1) s
h_0^{-\frac{1}{\beta}} \leq h_0^{\frac{\beta -2}{\beta}} \left[ 2s
h_0^{\frac{1}{\beta}} +\frac{\mu}{2\beta^2 L s^{2\beta -1}} \right].
$$
Since $\beta >2$, $\int_0^{h(\tau_1)} u(\tau_1, x;\phi) dx$ can be
as small as possible when $h_0 \rightarrow 0$. By a parallel
consideration, the same is true for $\int_{g(\tau_1)}^0 u(\tau_1,
x;\phi) dx$. This completes the proof.
\end{proof}

\vskip 6pt

We also need the following technical lemma.

\begin{lem}\label{lem:consequence of cp} Suppose
 $h_0 >0$ and $\phi \in \mathscr {X}(h_0)$. Then there exists
$\varepsilon_0 \in (0,h_0)$ such that, for any $\varepsilon \in (0,
\varepsilon_0)$, any $\phi_\varepsilon \in \mathscr {X}(h_0
-\varepsilon)$, and any sufficiently large $\sigma>0$,
$$
g_\varepsilon(t_1)\leq -h_0, \; h_\varepsilon (t_1)\geq h_0,\;
u_\varepsilon(t_1,x) \geq \phi(x) \quad \mbox{for } x\in [-h_0, h_0]
$$
at some $t_1 >0$, where $(u_\varepsilon, g_\varepsilon,
h_\varepsilon)$ denotes the solution of \eqref{p} with
$u_0=\sigma\phi_\varepsilon$.
\end{lem}

\begin{proof} We prove the conclusion by constructing a suitable
lower solution. Let $\varphi_1 (x)$ be the positive function
satisfying \eqref{SL} and  $\|\varphi_1\|_{L^\infty ([-1,1])}
=\varphi_1 (0)=1$.

Choose $\rho_0 >0$ such that
\begin{equation}\label{rho0}
\rho_0 > \frac{1}{ -2\mu \varphi'_1 (1)},\quad\ \  \rho_0 \varphi_1
\Big( \frac{x}{h_0} \Big) \geq \phi(x) \ \ \mbox{ for } x\in [-h_0,
h_0].
\end{equation}
This is possible since $\varphi_1'(1)<0$ and $\phi\in
C^1([-h_0,h_0])$. Fix such a $\rho_0$; then there exists
$M=M(\rho_0)>0$ such that
\begin{equation}\label{f>-Ms}
f(s) \geq - M s \quad \mbox{for } s\in [0, 2\rho_0].
\end{equation}

Set
\begin{equation}\label{def:lambda,epsilon0}
\lambda := \lambda_1 +M h_0^2,\quad \varepsilon_0 :=
(1-2^{-\frac{1}{2\lambda}}) h_0.
\end{equation}

For any $\varepsilon \in (0, \varepsilon_0)$, any $\phi_\varepsilon
\in \mathscr {X}(h_0 -\varepsilon)$, we will show that, when
$$
\sigma \phi_\varepsilon \geq \rho_0 \Big(\frac{h_0}{h_0
-\varepsilon}\Big)^{2\lambda} \cdot \varphi_1 \Big( \frac{x}{h_0
-\varepsilon} \Big),
$$
we have
$$
u(t_1, x;\sigma \phi_\varepsilon) \geq \phi(x)\quad \mbox{on }
[-h_0,h_0],
$$
at time $t_1 := 2\varepsilon h_0 - \varepsilon^2 >0$. To prove this
result, we first show that $(w, -k, k)$ given by
$$
w(t,x):= \rho_0 \Big(\frac{h_0}{k(t)}\Big)^{2\lambda} \cdot
\varphi_1 \Big( \frac{x}{k(t)} \Big),\quad k(t):= \sqrt{(h_0
-\varepsilon)^2 +t}
$$
forms a lower solution of \eqref{p} on the time interval $t\in [0,
t_1]$.

When $t\in [0,t_1]$, we have $h_0 -\varepsilon \leq k(t) \leq h_0$
and
$$
\|w(t,x)\|_{L^\infty ([-k(t),k(t)])} = w(t,0) \leq \rho_0 \Big(
\frac{h_0}{h_0 -\varepsilon} \Big)^{2\lambda} \leq 2\rho_0
$$
by the definition of $\varepsilon_0$. Hence, for $-k(t)\leq x \leq
k(t),\ 0\leq t\leq t_1$,
\begin{eqnarray*}
w_t -w_{xx}-f(w) &\leq & w_t -w_{xx} + Mw\\
& = & \frac{ -\rho_0 h_0^{2\lambda}}{[k(t)]^{2\lambda +2}} \left[
\varphi''_1 + \frac{x}{2k(t)} \varphi'_1 +[\lambda - M
(k(t))^2 ]\varphi_1 \right]\\
& \leq & \frac{ -\rho_0 h_0^{2\lambda}}{[k(t)]^{2\lambda +2}} \left[
\varphi''_1 + \frac{{\rm sgn}(x)}{2} \varphi'_1 + \lambda_1
\varphi_1 \right] =0,
\end{eqnarray*}
$$
k'(t)+\mu w_x(t,k(t)) \leq \frac{1}{2k(t)} \left[ 1+ 2\mu \rho_0
\Big(\frac{h_0}{k(t)} \Big)^{2\lambda} \varphi'_1 (1) \right] <0,
$$
and
$$
-k'(t)+\mu w_x(t,-k(t)) \geq -\frac{1}{2k(t)} \left[ 1+ 2\mu \rho_0
\Big(\frac{h_0}{k(t)} \Big)^{2\lambda} \varphi'_1 (1) \right] >0.
$$
If $\sigma$ is chosen such that $w(0,x) \leq  \sigma
\phi_\varepsilon$, we find that
  $(w(t,x),-k(t), k(t))$ is a lower solution of \eqref{p} with $u_0=\sigma \phi_\epsilon$.
  Now it is clear that the required inequalities follow from this and \eqref{rho0}.
\end{proof}

\bigskip

\noindent {\bf Proof of Propositions \ref{prop:bi-sigma-infty} and
\ref{prop:combus-sigma-infty}:}\ \ We only consider the (f$_B$)
case; the proof of the (f$_C$) case is identical.

By Lemma \ref{lem:int-small} and Theorem \ref{thm:vanishing} (iii),
we see that for sufficiently small $h_0
>0$, vanishing happens for any $\phi \in \mathscr {X}(h_0)$ and any $\sigma >0$.  So
$\sigma^* (h_0)=\infty$ for small $h_0$. Here and in what follows we
write $\sigma^*(h_0)=\infty$ instead of $\sigma^*(h_0,\phi)=\infty$,
since $\phi\in \mathscr{X}(h_0)$ plays no role for the validity of
this identity.

Define
\begin{equation}\label{def-ZB0}
Z^0_B := \sup \Pi\qquad \mbox{where } \ \Pi:= \{h_0 > 0 \mid
\sigma^* (h_0) = \infty\}.
\end{equation}
In view of the above fact and Proposition
\ref{prop:bi-sigma-finite}, we have $0< Z^0_B \leq Z_B$. By the
comparison principle, we see that $\sigma^*(h_0)=\infty$ when
$h_0\in (0, Z_B^0)$, that is, $(0,Z_B^0)\subset \Pi$.

We claim that the set $(0,\infty) \backslash \Pi$ is open, and so
$\Pi$ is closed. To see this, suppose $h_0$ belongs to this set and
so $\sigma^*(h_0, \phi)<\infty$ for every $\phi\in
\mathscr{X}(h_0)$. Hence there exists $\sigma_1>0$ so that spreading
happens when $u_0 = \sigma \phi$ and $\sigma\geq \sigma_1$. By Lemma
\ref{lem:consequence of cp}, for sufficiently small $\epsilon>0$ and
any $\phi_\epsilon\in \mathscr{X}(h_0-\epsilon)$, there exists
$\sigma_2>0$ and $t_1>0$ such that
\[
u(t_1,x; \sigma_2\phi_\epsilon)\geq \sigma_1\phi(x) \mbox{ in }
[-h_0,h_0].
\]
It follows that
\[
u(t+t_1,x;\sigma_2\phi_\epsilon)\geq u(t,x;\sigma_1\phi) \mbox{ for
all } t>0.
\]
This implies that $\sigma^*(h_0-\epsilon,\phi_\epsilon)<\infty$ and
hence $h_0-\epsilon\in (0,\infty) \backslash \Pi$. By the comparison
principle, clearly any $h>h_0$ belongs to this set. Thus it is an
open set.

Hence $\Pi$ is relatively closed in $(0,+\infty)$ and $\Pi=(0,
Z_B^0]$. By Proposition \ref{prop:bi-sigma-finite},
$\sigma^*(Z_B)<+\infty$. Therefore $Z_B^0<Z_B$, and for $h_0>Z_B^0$,
$\sigma^*(h_0)<+\infty$.

The proof is complete. \hfill $\square$

\vskip 6pt

The proof of Proposition \ref{prop:mono-sigma-infty}  needs the
following result.

\begin{lem}\label{M-to-0}
Under the assumptions of Proposition \ref{prop:mono-sigma-infty},
there exists $h^*_0 >0$ small such that, for any $h_0 \in (0,h^*_0)$
and any $\phi\in \mathscr {X}(h_0)$, we have $\|u(t,\cdot;
\phi)\|_\infty \rightarrow 0$ as $t\rightarrow \infty$.
\end{lem}

\begin{proof}
 By \eqref{cond alpha beta} we have
$$
\beta < \frac{2\beta^2 - 2\beta -1}{1+\beta}.
$$
Fix a constant $\alpha$ between them, we may assume without loss of
generality that
\begin{equation}\label{growth M1-2}
\limsup\limits_{s\rightarrow \infty} \frac{-f(s)}{s^{1+2\alpha}} <
\infty.
\end{equation}
Indeed, if \eqref{growth M1-2} does not hold, we can modify $f$ to
be $f_1 \in C^1$ such that $f(u) \leq f_1 (u)$ for $u\geq 0$ and
that $f_1$ satisfies \eqref{growth M1-2}. Replace $f$ by $f_1$ in
\eqref{p} and denote the problem by \eqref{p}$_1$. It is easily seen
that when a solution $u_1(t,x;\phi)$ of \eqref{p}$_1$ vanishes, the
solution $u(t,x;\phi)$ of \eqref{p} also vanishes. Hence we may
prove the lemma under the additional condition \eqref{growth M1-2}.

In what follows we always choose $h_0\in (0,1)$. Conditions
\eqref{growth M1} and \eqref{growth M1-2} imply that there exist $s
>1,\ K_\beta >0$ and $K_\alpha >0$ such that
\begin{equation}\label{growth M2}
K_\beta u^{1+2\beta} \leq -f(u) \leq K_\alpha u^{1+2\alpha} \qquad
\mbox{for } u\geq s.
\end{equation}
Moreover, we could have chosen $K_\beta=L$ given by \eqref{def-L}.
Therefore we can define $c$, $w(t,x)$ and $k(t)$ as in
 the proof of
Lemma \ref{lem:int-small} to deduce
\[
h(t)\leq 2h_0+ct \mbox{ for all } t>0.
\]
Similarly
\[
g(t)\geq -2h_0-ct \mbox{ for all } t>0.
\]

{\it Step 1. A bound from the proof of Lemma \ref{lem:int-small}}.
We denote $\omega_1 := \frac{1+2\beta}{2\beta^2} <1$; then as in the
proof of Lemma \ref{lem:int-small} we have
$$
\tau_1 := \int_{\frac{s}{h_0^{\omega_1}}}^{\infty} \frac{dr}{-f(r)}
\leq \frac{h_0^{2\beta \omega_1}} {2\beta K_\beta s^{2\beta}},
$$
$$
c\tau_1 \leq \frac{\mu}{2\beta^2 K_\beta s^{2\beta}} h_0,\quad\ \
\max\{-g(\tau_1), h(\tau_1)\} \leq 2h_0 +c\tau_1 \leq \Big(2+
\frac{\mu}{2\beta^2 K_\beta s^{2\beta}} \Big) h_0,
$$
and
\begin{equation}\label{int-small-M}
\int_{g(\tau_1)}^{h(\tau_1)} u(\tau_1, x; \phi)dx \leq \Big(4s +
\frac{\mu}{\beta^2 K_\beta s^{2\beta -1}} \Big) h_0^{1-\omega_1}.
\end{equation}

{\it Step 2. A bound for $g$ and $h$ at a later time $\tau_2$}. By
condition \eqref{cond alpha beta}, there exists $0<\delta <\beta$
such that
$$
\alpha <  \frac{2\beta^2 -2\beta -1}{1+\beta +\delta}.
$$
Hence
$$
1- \omega_1 =  \frac{2\beta^2 -2\beta -1}{2\beta^2} > \alpha
\omega_2 := \alpha \cdot \frac{1+\beta +\delta}{2\beta^2} .
$$
Note that $\omega_2 <\omega_1$ and so $h_0^{-\omega_2} <
h_0^{-\omega_1}$.

Define
$$
\tau_2 := \int_{\frac{s}{h_0^{\omega_2}}}^{\infty} \frac{dr}{-f(r)}
\leq \frac{h_0^{2\beta \omega_2}} {2\beta K_\beta s^{2\beta}} <
\tau^0_2 := \frac{1}{2\beta K_\beta s^{2\beta}}.
$$
Then
\begin{equation}\label{h02}
\max\{-g(\tau_2), h(\tau_2)\} \leq 2h_0 +c\tau_2 \leq 2h_0 +
\frac{\mu}{2\beta^2 K_\beta s^{2\beta}} h_0^{\frac{\delta}{\beta}}
\leq \frac{\pi}{3\sqrt{f'(0)}}
\end{equation}
provided $h_0$ is sufficiently small.

{\it Step 3. A key bound for $u$}. Direct calculation shows that
$$
\tau_2 - \tau_1 \geq
\int_{\frac{s}{h_0^{\omega_2}}}^{\frac{s}{h_0^{\omega_1}}}
\frac{dr}{K_\alpha r^{1+2\alpha} } = \frac{h_0^{2\alpha \omega_2}
(1- h_0^{2\alpha (\omega_1 -\omega_2)} )}{2\alpha K_\alpha
s^{2\alpha}} \geq \frac{h_0^{2\alpha \omega_2}}{4\alpha K_\alpha
s^{2\alpha}}
$$
provided that $h_0$ is sufficiently small such that
\begin{equation}\label{h03}
1-  h_0^{2\alpha (\omega_1 -\omega_2)} \geq \frac{1}{2}.
\end{equation}
Therefore,  for $g(\tau_2)\leq x\leq h(\tau_2)$, we have by
\eqref{int-small-M}
\begin{equation}\label{h04}
\frac{e^{K(\tau_2 -\tau_1)}}{2\sqrt{\pi (\tau_2 -\tau_1)}}
\int_{g(\tau_1)}^{h(\tau_1)} u(\tau_1, x;  \phi)dx  \leq
\widetilde{M} h_0^{1 -\omega_1 -\alpha \omega_2} < \sigma_1,
\end{equation}
provided $h_0 >0$ is sufficiently small, where $K>0$ is chosen such
that \eqref{cond2} holds, $\sigma_1
>0$ is small so that the conclusion in Theorem \ref{thm:vanishing}
(i) holds when $\|\phi\|_\infty\leq \sigma_1$, and
$$
\widetilde{M} :=  \frac{e^{K\tau_2^0}}{\sqrt{\pi}} \sqrt{4\alpha
K_\alpha s^{2\alpha}} \Big( 2s + \frac{\mu}{2\beta^2 K_\beta
s^{2\beta -1}} \Big).
$$

{\it Step 4. Completion of the proof}. For the above chosen $h_0
>0$,
$$
h(\tau_1) < h(\tau_2) < \frac{\pi}{3 \sqrt{f'(0)}},\; g(\tau_1) >
g(\tau_2) >- \frac{\pi}{3 \sqrt{f'(0)}}.
$$
By the proof of Lemma \ref{heat>} we know that
$$
u(\tau_1 +t, x; \phi) \leq \frac{e^{Kt}}{2\sqrt{\pi t}}
\int_{g(\tau_1)}^{h(\tau_1)} u(\tau_1, x;  \phi)dx\quad \mbox{for
all } t\geq 0.
$$
Hence for $ g(\tau_2)\leq x \leq h(\tau_2)$,
$$
u(\tau_2, x; \phi) \leq \frac{e^{K (\tau_2 -\tau_1)}}{2\sqrt{\pi
(\tau_2 -\tau_1)}} \int_{g(\tau_1)}^{h(\tau_1)} u(\tau_1, x; \phi)dx
< \sigma_1.
$$
Consequently, $u(\tau_2 +t, x;  \phi) \rightarrow 0$ by Theorem
\ref{thm:vanishing} (i).
\end{proof}

\vskip 6pt \noindent {\bf Proof of Proposition
\ref{prop:mono-sigma-infty}:}\ \ With the help of the above lemma,
one can proceed as in the proof of Propositions
\ref{prop:bi-sigma-infty} and \ref{prop:combus-sigma-infty}. \hfill
$\square$

\section{Semi-waves and spreading speed}\label{sec:speed}

Throughout this section we assume that $f$ is of type (f$_M$), or
(f$_B$), or (f$_C$), and $(u,g,h)$ is a solution of \eqref{p} for
which spreading happens. To determine the spreading speed, we will
construct suitable upper and lower solutions based on semi-waves and
waves of finite length with speed close to that of the semi-waves.

\subsection{Semi-waves}
\label{sec:semi-waves}

We call $q(z)$ a semi-wave with speed $c$ if $(c,q(z))$ satisfies
\begin{equation}\label{prob-q-infty}
\left\{
\begin{array}{l}
q'' -cq' +f(q)=0 \quad \mbox{ for } z\in (0,\infty),\\
q(0)=0, \ q(\infty)=1, \ q(z)>0\ \mbox{ for } z\in (0,\infty).
\end{array}
\right.
\end{equation}

As before, the first equation in \eqref{prob-q-infty} can be written
in the equivalent form
\begin{equation}
\label{q-p} q'=p,\; p'=cp-f(q).
\end{equation}
So a solution $q(z)$ of \eqref{prob-q-infty} corresponds to a
trajectory $(q(z), p(z))$ of \eqref{q-p} that starts from the point
$(0,\omega)$ $(\omega=q'(0)>0)$ in the $qp$-plane and ends at the
point $(1, 0)$ as $z\to+\infty$.

If $p(z)=q'(z)>0$ for all $z>0$, then the trajectory can be
expressed as a function $p=P(q),\; q\in [0,1]$, which satisfies
\begin{equation}
\label{P} \frac{dP}{dq} \equiv  P'=c-\frac{f(q)}{P} \; \mbox{ for }
q\in (0,1),\; P(0)=\omega,\; P(1)=0.
\end{equation}

It is easily checked that
$$
P_0 (q):= \sqrt{2\int_q^1 f(s)ds}, \quad  q\in [0,1],
$$
solves \eqref{P} with $c=0$ and $\omega=\omega^0:=\sqrt{2\int_0^1
f(s)ds}>0$. Moreover,
$$
P_0' (1) = -\sqrt{-f'(1)}.
$$

Suppose $c\geq 0$ and consider the equilibrium point $(1,0)$ of
\eqref{q-p}. A simple calculation shows that $(1,0)$ is a saddle
point, and hence by the theory of ODE (cf. \cite{Pet}) there are
exactly two trajectories of \eqref{q-p} that approach $(1,0)$  from
$q<1$;  one of them, denoted by $T_c$, has slope $\frac{c-\sqrt{c^2
- 4f'(1)}}{2}<0$ at $(1,0)$, and the other  has slope
$\frac{c+\sqrt{c^2 - 4f'(1)}}{2}>0$ at $(1,0)$. A part of $T_c$ that
lies in the set $S:=\{(q,p): 0\leq q\leq 1, p\geq 0\}$ and contains
$(1,0)$ is a curve which can be expressed as $p=P_c(q)$, $q\in [q_c,
1]$, where $q_c\in [0,1)$, $P_c(q)>0$ in $(q_c,1)$ and the point
$(q_c,P_c(q_c))$ lies on the boundary of $S$. Thus $P_c(q)$
satisfies
\begin{equation}
\label{P-c-1} P'=c-\frac{f(q)}{P} \mbox{ in $(q_c, 1)$},\; P(1)=0,\;
P'(1)=\frac{c-\sqrt{c^2 - 4f'(1)}}{2}.
\end{equation}
 Clearly, when $q_c>0$ we have
$P_c(q_c)=0$.

If $q_c>0$, then as $q$ is decreased from 1, $P_c(q)$  stays
positive and approaches $0$ from above as $q$ decreases to $q_c$.
Checking the sign of $P_c'(q)$ by the differential equation we
easily see that this cannot happen before $q$ reaches
$\tilde{\theta}$, where
$$
\tilde{\theta} = \left\{
 \begin{array}{ll}
  0, & \mbox{in } ({\rm f}_M)\ \mbox{case},\\
  \theta, & \mbox{in } ({\rm f}_B)\ \mbox{and } ({\rm f}_C) \mbox{ case}.
\end{array}
\right.
$$
Thus we always have $q_c\leq \tilde \theta$. For convenience of
notation, we assume that $P_c(q)=0$ for $q\in [0,q_c)$ when $q_c>0$,
so that $P_c(q)$ is always defined for $q\in [0,1]$. Denote
$T_c^1:=\{(q,p): p=P_c(q), q\in [0, 1]\}$.

Since
\[
0> \frac{c-\sqrt{c^2 - 4f'(1)}}{2}>-\sqrt{-f'(1)}=P_0'(1),
\]
we have $ P_c'(1)>P_0'(1)$, and by comparing the differential
equations of $P_c(q)$ and $P_0(q)$ we easily see that $P_c(q)$ never
touches $P_0(q)$ from below as $q$ decreases from 1 to $q_c$. Thus
\[
0<P_c(q) <P_0 (q)\quad \mbox{for } q\in (q_c, 1),
\]
which implies $P_c(q)<P_0(q)$ for $q\in (0,1)$.

In $(q_c, 1)$, we have
$$
 (P_c^2 - P^2_0)' = 2cP_c \leq 2cP_0 \leq
 2cM:=2c\|P_0\|_{L^\infty([0,1])}.
$$
Integrating this inequality over $[q, 1]\subset [q_c,1]$ we obtain
$$
P_0 (q) \geq P_c (q) \geq \sqrt{P^2_0 (q) -2cM (1-q)} \quad
\mbox{for } q\in (q_c, 1].
$$
This means that for sufficiently small $c>0$ we have $q_c=0$ and
$P_c(q)>0$ in $[0,1)$.

Define
\begin{equation}
\label{def-c0}
 c_0 := \sup \Lambda, \; \Lambda:=\{ \xi>0:  \ P_c(q)>0 \mbox{ in $[0,1)$
for all $c\in (0, \xi]$}\}.
\end{equation}
Then the above observation implies that $c_0\in (0,\infty]$. We
claim that
\begin{equation}
\label{c0-bound} c_0 \leq 2 \sqrt{K},\; \mbox{ where }
K:={\sup_{s>0} \frac{f(s)}{s}}.
\end{equation}
Since $f(u)\leq Ku$ for $u\geq 0$, we have
\[
P_c'\geq c-\frac{Kq}{P_c} \mbox{ in } (q_c, 1).
\]
If $c\geq 2\sqrt{K}$, then the linear function
$L(q)=\frac{c+\sqrt{c^2-4K}}{2}q$ satisfies
\[
L'=c-\frac{Kq}{L} \mbox{ for } q\in\R^1.
\]
It follows that $P_c(q) \leq L(q)$ in $(q_c, 1)$, which implies that
$P_c(q_c)=0$ and $c\not\in \Lambda$. Therefore $c_0\leq c$  for any
such $c$, and hence $c_0\leq 2\sqrt{K}$.

\begin{lem}
\label{p-c} For any $0\leq c_1 <c_2\leq c_0$ and $\bar{c}\geq 0$,
$$
P_{c_1} (q) > P_{c_2} (q)\; \mbox{in } [0,1), \;
\lim\limits_{c\rightarrow \bar{c}} P_c(q) = P_{\bar{c}}(q) \;\mbox{
uniformly in $[0,1]$.}
$$
Moreover,   $P_{c_0}(0)=0$ and $P_{c_0}(q)>0$ in $(0,1)$.
Furthermore, when $f$ is of $({\rm f}_B)$ or of $({\rm f}_C)$ type,
$q_c>0$ for $c>c_0$, and when $f$ is of $({\rm f}_M)$ type,
$P_c(0)=0,\; P_c(q)>0$ in $(0,1)$ for all $c\geq c_0$.
\end{lem}
\begin{proof} When $0\leq c_1<c_2$, from the formula for $P_c'(1)$
we find $P_{c_1}'(1)<P_{c_2}'(1)$. Since
\[
P_{c_1}'<c_2-\frac{f(q)}{P_{c_1}},
\]
we find that as $q$ decreases from $q=1$, the curve $p=P_{c_2}(q)$
remains below the curve $p=P_{c_1}(q)$. Therefore, $q_{c_2}\geq
q_{c_1}$ and  for $q\in (q_{c_2}, 1)$, $P_{c_1} (q) > P_{c_2} (q)$.
It follows that $P_{c}(q)$ is non-increasing in $c$ for $q\in
[0,1]$. Therefore for any $\bar c\geq 0$, as $c$ increases to $\bar
c$, $P_c(q)$ converges monotonically to some $ R(q)$ in $[0,1]$
{uniformly}. $R(q)$ represents a trajectory of \eqref{q-p} with
$c=\bar c$ that approaches $(1,0)$ from $q<1$, and its slope at
$(1,0)$ is negative. Therefore by the uniqueness of $T_{\bar c}$,
$R(q)$ must coincide with $P_{\bar c}(q)$. We can similarly show
that $P_c(q)$ converges to $P_{\bar c}(q)$ when $c$ decreases to
$\bar c$. Thus the curve $T_c^1$ varies continuously in $c$ for
$c\geq 0$.

If we assume further that $c_2< c_0$, then by the definition of
$c_0$ we know that $P_{c_i}(q)>0$ in $[0, 1)$ for $i=1,2$. Thus in
this case the above argument  yields $P_{c_1}(q)>P_{c_2}(q)$ for
$q\in [0,1)$.

We now consider $P_{c_0}(q)$. We must have $P_{c_0}(0)=0$, for
otherwise $P_{c_0}(0)>0$ which implies that $P_{c_1}(0)>0$ for
$c_1>c_0$ but close to $c_0$. However, this implies that $P_c(q)>0$
in $[0,1)$ for all $c\in (0, c_1]$ and thus $(0,c_1]\subset
\Lambda$. But this implies $c_0\geq c_1$, a contradiction. Thus we
always have $P_{c_0}(0)=0$.

To show that $P_{c_0}(q)>0$ in $(0,1)$, it suffices to prove that
$q_{c_0}=0$.
 Suppose by way of
contradiction that $q_{c_0}>0$. Since $q_{c_0}\leq\tilde\theta$, and
$\tilde \theta=0$ when $f$ is monostable, we find that $q_{c_0}>0$
cannot happen if $f$ is of (f$_M$) type.

Suppose that $f$ is of bistable type. Choose $\eta\in (0,
q_{c_0})\subset (0,\theta)$. Since $(\eta, 0)$ is a regular point
for \eqref{q-p}, there is a unique trajectory $T_{c,\eta}$ passing
through $(\eta, 0)$. Since $f(\eta)<0$, $T_{c,\eta}$ has a part in
$S$ that is a curve that can be expressed by $p=V_c(q)$, $q\in
[\eta, q^c]$ for some $q^c\in (\eta, 1]$, and $(q^c, V_c(q^c))$ lies
on the boundary of $S$, $V_c(q)>0$ in $(\eta, q^c)$,
\[
V_c'=c-\frac{f(q)}{V_c} \mbox{ in } (\eta, q^c).
\]
The curve $ p=V_{c_0}(q),\; q\in (\eta, q^{c_0})$,  is increasing
for $q\in (0,\theta)$ and it cannot intersect $T^1_{c_0}$. Hence it
remains above $T^1_{c_0}$. This implies that $q^{c_0}=1$. It cannot
join $(1,0)$ since $T_{c_0}$ is the only trajectory approaching this
point with a non-positive slope there. Therefore necessarily
$V_{c_0}(1)>0$. Thus this curve is a piece of trajectory of
\eqref{q-p} with $c=c_0$ that stays away from any equilibrium point.
Hence for all $c$ close to $c_0$, $V_c(q)$ stays close to
$V_{c_0}(q)$ in $[\eta, 1]$. In particular, for all $c<c_0$ close to
$c_0$, $V_c(q)>0$ in $(\eta, 1]$. This implies that for such $c$,
$T^1_c$ must lie below the curve $p=V_c(q) (\eta\leq q\leq 1)$. This
is impossible since by the definition of $c_0$, for such $c$,
$P_c(q)>0$ in $[0,1)$, which leads to $0=V_c(\eta)\geq P_c(\eta)>0$.

For the case that $f$ is of combustion type, the arguments need to
be modified, since now $(\eta, 0)$ is an equilibrium point of
\eqref{q-p}. Choose $\epsilon>0$ small so that $\eta-\epsilon
c^{-1}>0$. Then $(\eta,\epsilon)$ is a regular point of \eqref{q-p},
and hence there is a unique trajectory $T_{\eta,c,\epsilon}$ passing
through it. Since $f(u)=0$ in $(0,\theta]$, we see that the
trajectory is a straight line with slope $c$ near $(\eta,\epsilon)$,
and it intersects the $q$-axis at $(\eta-\epsilon c^{-1}, 0)$. Much
as in the bistable case above a piece of $T_{\eta,c,\epsilon}$ in
$S$ can be expressed as $p=\hat V_c(q)$, and $p=\hat V_{c_0}(q)$
lies above $T^1_{c_0}$ with $\hat V_{c_0}(1)>0$. We can now derive a
contradiction in the same way as in the bistable case. Thus we must
have $q_{c_0}=0$.

If $f$ is of (f$_M$) type, then for $c\geq  c_0$, $P_c(0)\leq
P_{c_0}(0)=0$. On the other hand, since $q_c=0$ we know that
$P_c(q)>0$ for $q\in (0, 1)$. Thus for such $c$, $P_c(0)=0$ and
$P_c(q)>0$ in $(0,1)$.

If $f$ is of (f$_B$) type, then $(0,0)$ is a saddle equilibrium
point of \eqref{q-p} for all $c\geq 0$, and from the ODE theory we
find that there are exactly two trajectories of \eqref{q-p} that
approach $(0,0)$ from $q>0$, one denoted by $T^0_c$ has slope
$\frac{c+\sqrt{c^2-4f'(0)}}{2}>0$ at $(0,0)$, the other has slope
$\frac{c-\sqrt{c^2-4f'(0)}}{2}<0$ at (0,0). For such $f$, if there
exists $c>c_0$ such that $q_c=0$, then we must have $P_c(0)=0$ for
otherwise, $P_c(0)>0$ and by the monotonicity of $P_c(q)$ on $c$, we
deduce $c_0\geq c$, contradicting to the choice of $c$. Thus
$P_c(0)=0$, and $p=P_c(q)$, $q\in[0,1]$, represents a trajectory of
\eqref{q-p} that connects $(0,0)$ and $(1,0)$. So it must coincide
with $T^0_c$ and hence $P'_c(0)=\frac{c+\sqrt{c^2-4f'(0)}}{2}>0$.
For the same reason we have
$P'_{c_0}(0)=\frac{c_0+\sqrt{c_0^2-4f'(0)}}{2}>0$. It follows that
$P'_c(0)>P_{c_0}'(0)$. On the other hand,  from
\[
P'_c(1)=\frac{c-\sqrt{c^2-4f'(1)}}{2} \mbox{ and }
P'_{c_0}(1)=\frac{c_0-\sqrt{c_0^2-4f'(1)}}{2}
\]
we deduce $P_c'(1)>P_{c_0}'(1)$. Thus there exists $q_*\in (0,1)$
such that $P_c(q)>P_{c_0}(q)$ in $(0, q_*)$ and
$P_c(q_*)>P_{c_0}(q_*)$. It follows that $P_c'(q_*)\leq
P_{c_0}'(q_*)$. However, from the differential equations we deduce
$P'_c(q_*)-P'_{c_0}(q_*)=c-c_0>0$. This contradiction shows that we
must have $q_c>0$ for $c>c_0$ in the (f$_B$) case.

If $f$ is of (f$_C$) type, and if $q_c=0$ for some $c>c_0$, then we
have $0\leq P_c(0)\leq P_{c_0}(0)=0$, and thus $P_c(0)=0$. We notice
from the differential equation for $P_c(q)$ that $P'_c(q)=c$  in
$(0,\theta]$ and hence $P_c(q)=cq$ in this range. For the same
reason $P_{c_0}(q)=c_0q$ in $(0,\theta]$. Thus we again have
$P_c'(0)>P'_{c_0}(0)$. We can now derive a contradiction as in the
(f$_B$) case above.

The proof of the lemma is now complete.
\end{proof}

\begin{thm}
\label{waves} Let $c_0$ and $P_{c_0}(q)$ be defined as above. Then
the trajectory represented by $p=P_{c_0}(q)$, $q\in (0,1)$, gives
rise to a solution $q_0(z)$ of the problem
\begin{equation}
\label{tw-c_0} \left\{
\begin{array}{l}
q'' -cq' +f(q)=0 \quad \mbox{ for } z\in \R^1,\\
q(-\infty)=0, \ q(\infty)=1, \ q(z)>0\ \mbox{ for } z\in \R^1,
\end{array}
\right.
\end{equation}
with $c=c_0$. Moreover, $q_0(z)$ is unique up to translation of the
variable $z$. This problem has no solution for any other nonnegative
value of $c$ if $f$ is of $({\rm f}_B)$ or of $({\rm f}_C)$ type,
and when $f$ is of $({\rm f}_M)$ type, it has a unique solution (up
to translation) for every $c\geq c_0$, and has no solution for $c\in
[0, c_0)$.

For each $\mu>0$, there exists a unique $c^*=c^*_\mu\in (0, c_0)$
such that $P_{c^*}(0)=\frac{c^*}{\mu}$. Moreover,
\eqref{prob-q-infty} has a unique solution $(c, q)=(c^*, q^*)$ satisfying $q'(0)=c/\mu$,
and $c^*_\mu$ is increasing in $\mu$ with
\[
\lim_{\mu\to\infty}c^*_\mu=c_0.
\]
\end{thm}
\begin{proof}
Let $(q_0(z), p_0(z))$, $z\in\R^1$, be the trajectory of \eqref{q-p}
corresponding to $p=P_{c_0}(q)$, $q\in (0,1)$. Then clearly $q_0(z)$
satisfies \eqref{tw-c_0} with $c=c_0$. Conversely, a solution of
\eqref{tw-c_0} gives rise to a function $P(q)$ satisfying \eqref{P}.
Thus $P(q)\equiv P_c(q)$. The conclusions about the existence and
nonexistence of  solutions to  \eqref{tw-c_0} now follow directly
from Lemma \ref{p-c}. The solution is unique up to translation
because the trajectory $T_{c_0}^1$ is the only one that approaches
$(1,0)$ from $q<1$ that has a negative slope there.

By Lemma \ref{p-c} and the definition of $c_0$, we find that for
each $c\in [0,c_0)$, $P_c(0)>0$ and it decreases continuously as $c$
increases in $[0, c_0]$. Moreover, $P_0(0)>0$ and $P_{c_0}(0)=0$. We
now consider the continuous function
\[
\xi(c)=\xi_\mu(c):=P_c(0)-\frac{c}{\mu},\; c\in [0,c_0].
\]
By the above discussion we know that $\xi(c)$ is strictly decreasing
in $[0, c_0]$. Moreover, $\xi(0)=P_0(0)>0$ and $\xi(c_0)=-c_0
/\mu<0$. Thus there exists a unique $c^*=c^*_\mu\in (0, c_0)$ such
that $\xi(c^*)=0$.

If we view $(c^*_\mu, c_\mu^*/\mu)$ as the unique intersection point
of the decreasing curve $y=P_c(0)$ with the increasing line
$y=c/\mu$ in the $cy$-plane, then it is clear that $c^*_\mu$
increases to $c_0$ as $\mu$ increases to $\infty$.

Finally the curve $p=P_{c^*}(q)$, $q\in [0,1)$, corresponds to a
trajectory of \eqref{q-p}, say $(q^*(z), p^*(z))$, $z\in
[0,\infty)$, that connects the regular point $(0, P_{c^*}(0))$ with
the equilibrium $(1,0)$. It follows from \eqref{q-p} with $c=c^*$
that $(c^*, q^*)$ solves \eqref{prob-q-infty} with $(q^*)'(0)=c^*/\mu$. If $(c,q)$ is another
solution of \eqref{prob-q-infty} satisfying $q'(0)=c/\mu$, then it corresponds to a
trajectory of \eqref{q-p} connecting $(0,c/\mu)$ and $(1,0)$ in the
set $S$. Since for each $c\geq 0$ there is only one such trajectory
joining $(1,0)$, it coincides with  $p=P_c(q)$, $q\in [0,1)$. Thus
we necessarily have $P_c(0)=c/\mu$ and hence $c=c^*$. It follows
that $q=q^*$.

The proof is complete.
\end{proof}

\smallskip

\begin{remark}
\label{wave-semiwave} The function $q_0(z)$ is usually called a
traveling wave with speed $c_0$. Its existence is well known. Our
proof of the existence of $(c_0, q_0(z))$ is somewhat different from
\cite{AW1, AW2}, so that our version of the proof can be easily used
to obtain the semi-wave $q^*(z)$ and to reveal the relationship
between $c^*$ and $c_0$. Since $(1,0)$ is always a saddle
equilibrium point of \eqref{q-p}, our construction of the connecting
orbit between $(0,0)$ and $(1,0)$,  based on the latter point, is
slightly simpler, compared with that in \cite{AW1, AW2}, where the
construction is based on $(0,0)$ instead.
\end{remark}

Proposition \ref{prop:semi-wave} clearly follows from Theorem
\ref{waves}.

\bigskip

Next we show how a suitable perturbation of the above setting can be
used to produce a semi-wave that can be used to construct upper
solutions for \eqref{p}. Let $\hat{\theta}\in (0,1)$ be the biggest
maximum point of $f$ in $(0,1)$. For  small $\varepsilon
>0$, let $f_\varepsilon (u)$ be a $C^1$ function obtained by modifying
$f(u)$ over $ [\hat{\theta} , 2]$ such that $f(u) \leq f_\varepsilon
(u)$ for $u\in \R^1$, $f_\varepsilon (u)$ has a unique zero
$1+\varepsilon$ in $[\hat{\theta}, 2]$, $f'_\varepsilon
(1+\varepsilon) <0$, and $f_\varepsilon$ decreases to $ f$ in the
$C^1$ norm over $[\hat\theta, 2]$ as $\varepsilon$ decreases to 0.

Replacing $f$ by $f_\varepsilon$, we have a parallel version of
Theorem \ref{waves}. We denote the corresponding wave and semi-wave
by $(c_0^\varepsilon, q_0^\varepsilon(z))$ and $(c^*_\varepsilon,
q^*_\varepsilon(z))$ respectively. We have the following result.

\begin{prop}
\label{q-perterb}
\[
c_0^\varepsilon\geq c_0,\; c^*_\varepsilon>c^*,\;
\lim_{\varepsilon\to 0}c_0^\varepsilon=c_0,\; \lim_{\varepsilon\to
0}c^*_\varepsilon=c^*.
\]
\end{prop}
\begin{proof}
Let $P_c^\varepsilon(q)$ denote the correspondent of $P_c(q)$. Since
$f_\varepsilon \geq  f$,  as $q$ decreases from $1+\varepsilon$, the
curve $p=P_c^\varepsilon(q)$ cannot touch the curve $p=P_c(q)$ from
above. As before, it cannot touch $p=0$ before $q$ reaches
$\tilde\theta$. Therefore $P_c^\varepsilon(q)$ is positive over
$(\tilde\theta, 1+\varepsilon)$ and $P_c^\varepsilon(q)>P_c(q)$ in
$[q_c,1)$. Thus for $c\in (0, c_0)$, $P_c^\varepsilon(0)>P_c(0)>0$.
This implies that $c_0^\varepsilon\geq c_0$.

Using the monotonicity of $f_\varepsilon$ on $\varepsilon$, we
easily deduce that $P_c^\varepsilon(q)$ is non-decreasing in
$\varepsilon$. Thus $P_c^\varepsilon(q)$ converges to some $ R(q)$
as $\varepsilon\to 0$ uniformly in $[0,1]$. Since $p=R(q)$
represents a trajectory of \eqref{q-p} that approaches $(1,0)$ with
a non-positive slope at $(1,0)$, and there is only one such
trajectory, $R(q)$ must coincide with $P_c(q)$. In particular, we
have $P^\varepsilon_c(0)\to P_c(0)$ as $\varepsilon\to 0$.

In view of the definition of $c^\varepsilon_0$, the monotonicity of
$P_c^\varepsilon(q)$ on $\varepsilon$ implies that $c^\varepsilon_0$
is nondecreasing in $\varepsilon$. Therefore $\hat
c_0:=\lim_{\varepsilon\to 0} c_0^\varepsilon$ exists and $\hat
c_0\geq c_0$.

Suppose $\hat c_0>c_0$, we are going to deduce a contradiction.
Choose $c\in (c_0, \hat c_0)$ and consider $P_c^\varepsilon(q)$.
Since $c<c_0^\varepsilon$, we have $P_c^\varepsilon(q)>0$ in
$[0,1+\varepsilon)$ for all $\varepsilon$.

Then in the case that $f$ is of (f$_B$) type or of (f$_C$) type, we
have $ (P_c^\varepsilon(q))' \geq c$ in $(0, \theta]$ and hence
$P_c^\varepsilon(q)\geq cq$ in $[0,\theta]$. Letting $\varepsilon\to
0$, we deduce $P_c(q)\geq cq$ in $[0,\theta]$. We already know from
the proof of Lemma \ref{p-c} that $P_c(q)>0$ in $(q_c,1)\supset
(\theta, 1)$. Thus $P_c(q)>0$ in $(0,1)$. If $P_c(0)>0$ then by the
monotonicity in $c$ we have $P_{c'}(0)>0$ for all $c'\in (0, c]$ and
hence $c_0\geq c$, a contradiction to our choice of $c$. If
$P_c(0)=0$, then $p=P_c(q)$, $q\in (0,1)$, represents a trajectory
of \eqref{q-p} connecting $(0,0)$ and $(1,0)$. By Theorem
\ref{waves} such a trajectory exists only if $c=c_0$, so we again
reach a contradiction.

Thus we have proved that $\lim_{\varepsilon\to
0}c_0^\varepsilon=c_0$ when $f$ is of (f$_B$) type or of (f$_C$)
type.

We now consider the case that $f$ is of (f$_M$) type. Suppose that
$\hat c_0>c_0$ and fix $c\in (c_0,\hat c_0)$. Note that from the
monotonicity of $c^\varepsilon_0 $ in $\varepsilon$, we always have
$c^\varepsilon_0\geq \hat c_0
>c$. Moreover, from the differential equation we easily see that as $\varepsilon$ decreases to $0$,
$P_c^\varepsilon(q)$ decreases to $P_c(q)$ uniformly in $[0,1]$, and
$P_{c}(q)<P_{c_1}(q)$ in $(0,1)$ if $c>c_1>c_0$. We fix such a
$c_1$.  Thus for sufficiently small $\varepsilon>0$,
$P_{c}^\varepsilon(\hat\theta)<P_{c_1}(\hat\theta)$. We now consider
$P_c^\varepsilon(q)$ for $q\in [0,\hat\theta]$. We notice that in
this range $f_\varepsilon(q)=f(q)$, and thus $P_c^\varepsilon(q)$
satisfies
\[
P'=c-\frac{f(q)}{P}
\]
for $q\in (0,\hat\theta]$. Since
$P_{c}^\varepsilon(\hat\theta)<P_{c_1}(\hat\theta)$, the curve
$p=P_{c}^\varepsilon(q)$ remains below the curve $p=P_{c_1}(q)$ as
$q$ is decreased from $q=\hat\theta$. Thus, due to $P_{c_1}(0)=0$
(because $c_1>c_0$), we necessarily have $P_c^{\varepsilon}(0)=0$.
On the other hand, due to $c<c_0^\varepsilon$, we must have
$P_c^\varepsilon(0)>0$. This contradiction  shows that $\hat
c_0=c_0$ in the monostable case as well.

For $c\in (0, c_0)$, since $P_{c}^\varepsilon(0)>P_c(0)$, we have
$\xi_\varepsilon(c):=P_c^\varepsilon(0)-c/\mu>\xi(c):=P_c(0)-c/\mu$.
It follows that $\xi(c^*)=0<\xi_\varepsilon(c^*)$, which implies
$c^*_\varepsilon>c^*$ since $\xi_{\varepsilon}(c)$ is strictly
decreasing in $c$. Since  $P^\varepsilon_{c} (0)$ is non-decreasing
in $\varepsilon$, we  deduce that $c^*_\varepsilon$ is
non-decreasing in $\varepsilon$.  The fact that $c^*_\varepsilon\to
c^*$ as $\varepsilon\to 0$ now follows easily from the uniqueness of
$c^*$ as a solution of $\xi(c)=0$.

The proof is now complete.
\end{proof}

Finally we show how a semi-wave can be perturbed to give a wave of
finite length which is more convenient to use in the construction of
lower solutions for \eqref{p}. So let $(c^*, q^*)$ be the unique
solution to \eqref{prob-q-infty}. Denote $\omega^*:=c^*/\mu$ and for
each $c\in (0, c^*)$ consider the problem
\begin{equation}
\label{perturb-semiwave} P'=c-\frac{f(q)}{P},\; P(0)=\omega^*.
\end{equation}
Since $c<c^*$, we easily see that the unique solution $P^c(q)$ of
this problem stays below $P_{c^*}(q)$ as $q$ increases from 0.
Therefore there exists some $Q^c \in (0,1]$ such that $P^c(q)>0$ in
$[0,Q^c)$ and $P^c (Q^c) =0$. We must have $Q^c<1$ because otherwise
we would have $P^c(q)\equiv P_c(q)$ due to the uniqueness of the
trajectory of \eqref{q-p} that approaches $(1,0)$ from $q<1$ with a
non-positive slope there, but this is impossible since $P_c(0)
>P_{c^*}(0)=\omega^*=P^c(0)$. It is also easily seen that, as $c$
increases to $c^*$, $Q^c$ increases to 1 and $P^c(q)\to P_{c^*}(q)$
uniformly, in the sense that $\|P^c-P_{c^*}\|_{L^\infty([0,
Q^c])}\to 0$. Let $(q^c(z), p^c(z))$ denote the trajectory of
\eqref{q-p} represented by the curve $p=P^c(q)$, $q\in [0, Q^c]$,
with $(q^c(0), p^c(0))=(0,\omega^*)$ and $(q^c(z^c), p^c(z^c))=(Q^c,
0)$, then clearly $q^c(z)$ solves \eqref{prob-q} with $Z=z^c$.
Moreover,  we have
\begin{equation}\label{c<c*}
c <c^* =  \mu \omega^* = \mu \, (q^c)'(0)
\end{equation}
and
\begin{equation}
\label{q^c-q^*}\lim_{c\nearrow\, c^*} z^c=+\infty,\;
\lim_{c\nearrow\, c^*}\|q^c-q^*\|_{L^\infty([0, z^c])}=0.
\end{equation}

\subsection{Asymptotic spreading speed}

Let $(c^*, q^*(z))$ be given as in Theorem \ref{waves}. For $c\in
(0, c^*)$, let $q^c(z)$, $Q^c$ and $z^c$ be as in the previous
subsection. For $t\geq 0$ we define
\[
k(t)=k_c(t):=z^c+ct\] and
\[
w(t,x)=w_c(t,x):=\left\{\begin{array}{ll} q^c(k(t)-x), & x\in [ct, k(t)],\\
q^c(z^c), & x\in [-ct, ct],\\
q^c(k(t)+x), & x\in [-k(t), -ct].
\end{array}
\right.
\]
We will use $(w, -k, k)$ as a lower solution to \eqref{p} in the
proof of the following result.

\begin{lem}\label{lem:u-to-1} Let $(u, g, h)$ be a solution of
\eqref{p} for which spreading happens. Then for any $c\in (0,c^*)$
and any $\delta\in (0, -f'(1))$, there exist positive numbers $T_*$
and $ M$  such that for $t\geq T_*$,
\begin{itemize}
\item[\rm (i)]
$
[g(t), h(t)]\supset [-ct, ct];
$
\item[\rm (ii)]
$ u(t,x)\geq 1-Me^{-\tilde\delta t}\quad \mbox{for } x\in [-ct, ct] \mbox{ and some } \tilde\delta=\tilde\delta(c)\in (0,\delta); $
\item[\rm (iii)]
$ u(t,x) \leq 1 + M e^{-\delta t} \quad \mbox{for } x \in [g(t),
h(t)]. $
\end{itemize}
\end{lem}

\begin{proof} (i) Fix $\hat c\in (c, c^*)$. Since spreading
happens we can find $T_1>0$ such that
\[
[g(T_1), h(T_1)]\supset [-k_{\hat c}(0), k_{\hat c}(0)] \mbox{ and }
u(T_1,x)>w_{\hat c}(0, x) \mbox{ in } [-k_{\hat c}(0), k_{\hat
c}(0)] .
\]
One then easily checks that $(w_{\hat c}(t-T_1,x),-k_{\hat
c}(t-T_1), k_{\hat c}(t-T_1))$ is a lower solution of \eqref{p} for
$t\geq T_1$. Hence for $t\geq T_2$ with some $T_2>T_1$,
\[
g(t)\leq -k_{\hat c}(t-T_1)<-\hat{c}(t-T_1)<-ct,\; h(t)\geq k_{\hat
c}(t-T_1)>\hat{c}(t-T_1)>ct
\]
 and
\[
u(t,x)\geq w_{\hat c}(t-T_1,x) \mbox{ for } x\in [-k_{\hat
c}(t-T_1), k_{\hat c}(t-T_1)]\supset [-ct, ct].
\]

(ii) Since $w_{\hat c}(t-T_1,x)\equiv q^{\hat c}(z^{\hat c})=Q^{\hat
c}>Q^c$ for $|x|\leq ct < \hat c(t-T_1)$ for all $t\geq T_2$, we
find from the above estimate for $u$ that
$$
u(t,x) \geq Q^{ c} \quad \mbox{for } -ct\leq x \leq ct,\; t\geq T_2.
$$

Since $f'(1) <0$, for any $\delta\in (0, -f'(1))$ we can find
$\rho=\rho(\delta)\in (0, 1)$ such that
\begin{equation}\label{delta}
f(u) \geq \delta (1-u) \ \ (u\in [1-\rho , 1]),\qquad f(u) \leq
\delta (1-u) \ \ (u\in [1, 1+\rho]).
\end{equation}
 Recall that
$Q^c\to 1$ as $c$ increases to $c^*$. Without loss of generality we
may assume that $c$ has been chosen so that $Q^c>1-\rho$.

Fix $T\geq T_2$ and let $\psi$ be the solution of
\begin{equation}\label{prob-psi-speed-1}
\left\{
 \begin{array}{ll}
  \psi_t = \psi_{xx} -\delta (\psi-1), & - cT <x< cT,\ t>0,\\
  \psi(t, \pm cT) \equiv Q^c, & t>0,\\
  \psi (0,x) \equiv Q^c, & - cT \leq x\leq cT.
  \end{array}
 \right.
\end{equation}
Since $\underline \psi\equiv Q^c$ is a lower solution of the
corresponding elliptic problem of \eqref{prob-psi-speed-1},  and
$\overline \psi\equiv 1$ is an upper solution, $\psi (t,x)$
increases in $t$ and $\psi \in [Q^c, 1]$. Moreover, $\psi$ is a
lower solution for the equation satisfied by $u(t+T,x)$ in the
region  $(t,x)\in [0,\infty)\times [-cT, cT]$, and so
\begin{equation}\label{psi-u}
\psi(t,x) \leq u(t+T,x) \quad \mbox{for } -cT\leq x\leq cT,\ t\geq
0.
\end{equation}

Set $\Psi := (\psi -Q^c) e^{\delta t}$, then
\begin{equation}\label{prob-Psi-speed}
\left\{
 \begin{array}{ll}
  \Psi_t = \Psi_{xx} +\delta (1-Q^c) e^{\delta t}, & - cT <x< cT,\ t>0,\\
  \Psi(t, \pm cT) \equiv 0, & t>0,\\
  \Psi (0,x) \equiv 0, & - cT\leq x\leq cT,
  \end{array}
 \right.
\end{equation}
The Green function of this problem can be expressed in the form (see page 84 of \cite{Fr})
$$
\widetilde{G}(t,x) = \sum\limits_{n\in \mathbb{Z}} (-1)^n G(t, x-2n
cT), $$ which yields
$$
 \widetilde{G}(t,x)\geq \widehat{G}(t,x) := G(t,x)-G(t, x-2 cT) -
G(t, x+2cT),
$$
where $G$ is the fundamental solution of the heat equation:
$$
G(t,x) = \frac{1}{\sqrt{4\pi t}} e^{-\frac{x^2}{4t}}.
$$
Hence
\begin{eqnarray*}
\Psi(t,x) & = & \int_0^t d\tau \int_{-cT}^{cT} \widetilde{G}
(t-\tau, x-\xi) \delta (1-Q^c) e^{\delta (t-\tau)} d\xi\\
& \geq & \delta (1-Q^c) \int_0^t e^{\delta (t- \tau)} d\tau
\int_{-cT}^{cT} \widehat{G}(t-\tau, x-\xi) d\xi
\end{eqnarray*}

For any $\epsilon \in (0,1)$, consider $(t,x)$ satisfying
\begin{equation}\label{t-x-range}
|x| \leq (1-\epsilon) cT, \quad 0<t \leq \frac{\epsilon^2 c^2 T}{4}.
\end{equation}
For such $(t,x)$ and any $\tau\in (0,t)$, we have
\begin{equation}\label{t-x-range-2}
\frac{cT \pm x}{2\sqrt{t-\tau}} \geq \frac{\epsilon
cT}{2\sqrt{t-\tau}} \geq \frac{\epsilon cT}{2\sqrt{t}} = \sqrt{T}
\cdot \frac{\epsilon c \sqrt{T}}{2\sqrt{t}} \geq \sqrt{T} \geq 1.
\end{equation}
So for $(t,x)$ satisfying \eqref{t-x-range} we have
$$
\int_{-cT}^{cT} G(t-\tau, x-\xi) d\xi = \left( \int_{-\infty}^\infty
- \int_{-\infty}^{-cT} - \int_{cT}^\infty \right) G(t-\tau,x-\xi)
d\xi = 1- I_1 -I_2
$$
where
$$
I_1 := \frac{1}{\sqrt{\pi}} \int_{-\infty}^{-\frac{cT
+x}{2\sqrt{t-\tau}}} e^{-r^2} dr,\; I_2 := \frac{1}{\sqrt{\pi}}
\int_{\frac{cT -x}{2\sqrt{t-\tau}}}^{\infty} e^{-r^2} dr.
$$
Using the elementary inequality
\[
\int_y^\infty e^{-r^2}dr\leq \int_y^\infty r e^{-r^2/2}dr=e^{-y^2/2}
\mbox{ for all } y\geq 1,
\]
 $(cT\pm x)\geq \epsilon c T$, and \eqref{t-x-range-2}, we deduce
\[
I_1,\; I_2\leq \frac{1}{\sqrt{\pi}}e^{-\frac{(\epsilon c
T)^2}{8(t-\tau)}}.
\]
But \eqref{t-x-range-2} also infers
\[
\frac{(\epsilon c T)^2}{8(t-\tau)}\geq T/2.
\]
Thus
\[
I_1,\; I_2\leq \frac{1}{\sqrt{\pi}}e^{-T/2},
\]
and
$$
\int_{-cT}^{cT} G(t-\tau, x-\xi) d\xi \geq 1-
\frac{2}{\sqrt{\pi}}e^{-T/2}.
$$
Similarly,
\begin{eqnarray*}
 \int_{-cT}^{cT} G(t-\tau, x-\xi - 2cT) d\xi  & = & \frac{1}{\sqrt{\pi}}
\int_{- cT}^{cT} \frac{1}{2\sqrt{t-\tau}} e^{-\frac{(x-\xi - 2 cT)^2}{4(t-\tau)}} d\xi \\
& \leq &  \frac{1}{\sqrt{\pi}} \int_{- cT}^\infty
\frac{1}{2\sqrt{t-\tau}} e^{-\frac{(x-\xi -2 cT)^2}{4(t-\tau)}} d\xi \\
& = &  \frac{1}{\sqrt{\pi}} \int_{\frac{cT
-x}{2\sqrt{t-\tau}}}^\infty e^{-r^2} d r  \leq \frac{1}{\sqrt{\pi}}
e^{- T/2}.
\end{eqnarray*}
Consequently, for $(t,x)$ satisfying \eqref{t-x-range}, we have
$$
\Psi(t,x) \geq \delta(1-Q^c) \int_0^t  e^{\delta (t-\tau)}  \left(
1- \frac{4}{\sqrt{\pi}} e^{-T/2} \right) d\tau = (1-Q^c) \left[ 1-
\frac{4}{\sqrt{\pi}} e^{-T/2} \right] (e^{\delta t } -1).
$$
This implies that, for such $(t,x)$,
$$
\psi (t,x) \geq  1- \frac{4}{\sqrt{\pi}} e^{-T/2} -e^{-\delta t}.
$$
 Taking $t=\frac{\epsilon^2c^2}{4}T$ we obtain
$$
\psi \Big( \frac{\epsilon^2 c^2}{4}{T},x \Big) \geq 1-
\frac{4}{\sqrt{\pi}} e^{-T/2} - e^{-\epsilon^2c^2\delta T/4}.
$$
We only focus on small $\epsilon>0$ such that $\epsilon^2 c^2 \delta
<2$, so
$$
\psi \Big( \frac{\epsilon^2 c^2}{4}{T},x \Big) \geq 1- M_0 e^{-
\epsilon^2c^2\delta T/4}\quad \mbox{ with } M_0:=
\frac{4}{\sqrt{\pi}}+1
$$
for $|x| \leq (1-\epsilon) cT$ and $ T \geq T_2$.

By \eqref{psi-u}, for such $T$ and $x$, we have
$$
u\Big( \frac{\epsilon^2 c^2}{4}{T} +T, x\Big) \geq 1-M_0e^{-
\epsilon^2c^2\delta T/4}.
$$
Finally, if we rewrite
$$
t = \frac{\epsilon^2 c^2}{4}{T}+T,
$$
then
$$
T = \left(1+\frac{\epsilon^2c^2}{4}\right)^{-1}t.
$$
Thus
$$
u(t,x) \geq 1- M_0  e^{-\tilde\delta t} \quad \mbox{for }
|x| \leq (1-\epsilon)\left(1+\frac{\epsilon^2c^2}{4}\right)^{-1}ct,\ t\geq T_3,
$$
where $\tilde\delta:=\frac{\epsilon^2c^2}{4}\left(1+\frac{\epsilon^2c^2}{4}\right)^{-1}\delta$ and $T_3 := \frac{\epsilon^2 c^2}{4}T_2+T_2$. Since this is true
for any $c\in (0, c^*)$ close to $c^*$, and any small $\epsilon>0$,
the above estimate implies the conclusion in (ii).

(iii) Consider the equation $\eta'(t) = f(\eta)$ with initial value
$\eta (0)=\|u_0\|_{L^\infty}+1$. Then $\eta$ is an upper solution of
\eqref{p}. So $u(t,x) \leq \eta(t)$ for all $t\geq 0$. Since
$f(u)<0$ for $u>1$, $\eta(t)$ is a decreasing function converging to
1 as $t\to\infty$. Hence there exists $T_4>0$ such that
$\eta(t)<1+\rho$ for $t\geq T_4$. Now, for $t\geq T_4$,  $\eta' (t)
=f(\eta) \leq \delta (1-\eta)$, and so
$$
u(t,x) \leq \eta (t) \leq 1 +\rho e^{-\delta (t-T_4)} \quad
\mbox{for } g(t)\leq x\leq h(t),\ t\geq T_4.
$$
This completes the proof. \end{proof}

\noindent
 {\bf Proof of Theorem \ref{thm:spreading speed}}. Assume
that spreading happens. Then for any given small $\varepsilon>0$, we
can apply Lemma \ref{lem:u-to-1} to obtain some $T_1>0$ large such
that for $t\geq T_1$,
\begin{equation}
\label{estimate-1} [g(t),h(t)]\supset [-(c^*-\varepsilon)t,
(c^*-\varepsilon)t] \mbox{ and } |u(t,x)-1|\leq Me^{-\delta t}
\mbox{ for } |x|\leq (c^*-\varepsilon)t.
\end{equation}

 We now make use of the perturbation method introduced in the previous subsection. For
  small
$\varepsilon_1
>0$, we  modify $f$ to obtain $f_{\varepsilon_1}$ and
$(c^*_{\varepsilon_1}, q^*_{\varepsilon_1})$ as described there.
Since $c^*_{\varepsilon_1}\to c^*$ as $\varepsilon_1\to 0$, we can
choose $\varepsilon_1>0$ small enough such that $c^*_{\varepsilon_1}
< c^* +\varepsilon$.

By Lemma \ref{lem:u-to-1} we see that for some large $T_2
>0$, $u(t ,x) < 1+ \varepsilon_1/2$ for $t\geq T_2$. We then choose $M'
>0$ large enough such that
\[
-c^*_{\varepsilon_1} T_2 - M' <g(T_2),\;\; c^*_{\varepsilon_1} T_2 +
M'
>h(T_2),
\]
 and
$$
 1+ \varepsilon_1/2  < q^*_{\varepsilon_1} (c^*_{\varepsilon_1}
T_2 + M' -x)\quad \mbox{for } x\in [g(T_2), h(T_2)].
$$
Therefore if we define
$$
k(t):=c^*_{\varepsilon_1} t +M',\; w(t,x):=q^*_{\varepsilon_1} (k(t)
-x),
$$
then $(w,g,k)$ is an upper solution of \eqref{p} for $t\geq T_2$,
and we can use Lemma \ref{lem:comp2} to deduce that
$$
h(t) \leq c^*_{\varepsilon_1} t + M' < (c^* +\varepsilon) t +M'
\quad \mbox{for } t\geq T_2.
$$

We can similarly show that
$$
g(t) \geq -(c^* +\varepsilon) t - M'   \quad \mbox{for } t\geq T_2.
$$
These estimates and \eqref{estimate-1} clearly imply
\[
\lim_{t\to\infty}\frac{-g(t)}{t}=\lim_{t\to\infty}\frac{h(t)}{t}=c^*.
\]
 This
completes the proof of Theorem \ref{thm:spreading speed}. \hfill
$\square$

\end{document}